\documentclass[a4paper,12pt]{article}
\usepackage[a4paper]{geometry}
\usepackage[utf8]{inputenc}
\usepackage{amsmath}
\usepackage{amsthm}
\usepackage{amssymb}
\usepackage{arydshln}
\usepackage{lscape}
\usepackage{rotating}
\usepackage{bm}
\usepackage{graphicx}
\newcommand{\GL}{\ensuremath \mathrm{GL}}
\newcommand{\GO}{\ensuremath \mathrm{GO}}
\newcommand{\SO}{\ensuremath \mathrm{SO}}
\newcommand{\Sp}{\ensuremath \mathrm{Sp}}
\newcommand{\CSp}{\ensuremath \mathrm{CSp}}
\newcommand{\SL}{\ensuremath \mathrm{SL}}
\newcommand{\St}{\ensuremath \mathrm{St}}
\newcommand{\tr}{\ensuremath \mathrm{tr}}
\newcommand{\reg}{\ensuremath \mathrm{reg}}
\newcommand{\F}{\ensuremath \mathbb{F}}

\newcommand{\Irr}{\ensuremath \mathrm{Irr}}
\newcommand{\Ocal}{\ensuremath \mathcal{O}}
\newcommand{\Z}{\ensuremath \mathbb{Z}}

\newcommand{\C}{\ensuremath \mathbb{C}}
\makeatletter
\newcommand{\maop}[1]{%
\ensuremath{\mathop{\operator@font #1}\nolimits}}
\newcommand{\maopl}[1]{%
\ensuremath{\mathop{\operator@font #1}\limits}}
\makeatother

\newcommand{\infl}{\maop{Infl}}

\newcommand{\G}{\bm{G}}
\newcommand{\T}{\bm{T}}
\newcommand{\W}{\bm{W}}

\newcommand{\R}[2]{\maop{R}^{#1}_{#2}}

\newcommand{\mbv}{\ensuremath \mathbf{v}}

\newcommand{\mbx}{\ensuremath \mathbf{x}}
\newcommand{\GAP}{\textsf{GAP}}
\newcommand{\CHEVIE}{\textsf{CHEVIE}}

\theoremstyle{plain}
\newtheorem{theorem}{Theorem}[section]

\newtheorem{lemma}[theorem]{Lemma}
\newtheorem{corollary}[theorem]{Corollary}

\theoremstyle{definition}

\theoremstyle{remark}
\newtheorem{remark}[theorem]{Remark}

\newtheorem{hypothesis}[theorem]{Hypothesis}

\date{\today}
\author{Frank Himstedt and Felix Noeske}
\title{Decomposition numbers of $\SO_7(q)$ and $\Sp_6(q)$}
\begin{document}
\maketitle
\begin{abstract}
We complete the $\ell$-modular decomposition numbers of the unipotent
characters in the principal block of the special orthogonal groups
$\SO_7(q)$ and the symplectic groups $\Sp_6(q)$ for all prime
powers $q$ and all odd primes~$\ell$ different from the defining
characteristic.
\end{abstract}

\section{Introduction}

In the representation theory of finite groups, the $\ell$-modular
decomposition numbers link the ordinary representations of a group to
its representations over fields of prime characteristic $\ell$. The
decomposition numbers of finite groups of Lie type of small rank in
the case that $\ell$ differs from the defining characteristic were
computed by several authors; see for example \cite{Dudas,
GeckSU3, GeckDissPub, HimHua3D4, Him2F4, HissG2, LandrockMichler, 
OkuyamaWaki, OkuWakiSU3, Waki_G2, WhiteSp4oddchar2, WhiteSp4odd,
WhiteSp4even, WhiteSp6even}.

Finite groups of Lie type which are dual to each other have many
properties in common. For example, they have isomorphic Weyl
groups. It is a natural question to what extent these similarities
carry over to modular representations. The smallest example 
(in terms of the rank) of two untwisted groups in duality with
non-isomorphic root systems are the special orthogonal groups
$\SO_7(q)$ and the symplectic groups $\Sp_6(q)$, having root systems
of type $B_3$ and $C_3$, respectively. This note is a contribution to
the modular representation theory in non-defining characteristic of
these groups. Since one expects, roughly speaking, that the
representation theory of a given block of a finite group of Lie type
may be reduced to the representation theory of unipotent blocks (see
\cite{BonnafeRouquier} and \cite{Enguehard} for important results in
this direction) the unipotent blocks are of particular relevance. 

Let $q$ be a power of a prime $p$ and $\ell$ an odd prime different
from $p$. We give the $\ell$-modular decomposition
matrices of the unipotent blocks of the groups $\SO_7(q)$ and $\Sp_6(q)$.
By the theory of basic sets \cite{GeckHissBasicSets} the
decomposition matrices can be reconstructed from the decomposition
numbers of the ordinary unipotent irreducible characters via so-called
basic relations. Only the principal block for $\ell \mid q+1$
leads to new results: all other blocks and other primes $\ell$ yield
cyclic blocks which have already been treated by Fong and Srinivasan 
\cite{FongSrinivasanTrees} and White \cite{WhiteSp6even}, or in
the case of $\ell \mid q-1$ results by Puig \cite{Puig} and
Gruber and Hiss \cite{GruberHiss} readily give the decomposition numbers. 

The $\ell$-modular decomposition matrices of the unipotent blocks of
$\SO_7(q)$ and $\Sp_6(q)$ are completed by Theorem~\ref{thm:decnumbers}, 
the main result of this paper. It gives all $\ell$-modular
decomposition numbers of the ordinary unipotent irreducible characters
in the principal block of $\SO_7(q)$ and $\Sp_6(q)$ for $\ell \mid q+1$. 
It turns out that, with respect to a suitable ordering of rows and
columns, the decomposition matrices of the unipotent blocks of
$\SO_7(q)$ and $\Sp_6(q)$ coincide. Furthermore, 
Theorem~\ref{thm:decnumbers} confirms \cite[Conjecture~3.4]{GeckHiss} 
in the special case of $\SO_7(q)$ and $\Sp_6(q)$.

It should be noted that for the symplectic groups $\Sp_6(q)$
most of the decomposition numbers were previously computed
by K\"ohler~\cite{DissKoehler}, An, Hiss~\cite{AnHissSteinberg} and 
White~\cite{WhiteSp6even} except for two parameters in the Steinberg
character. Theorem~\ref{thm:decnumbers} gives the values of these
parameters. 

The paper is structured as follows: We begin by detailing in
Section~\ref{sec:symporth} the groups considered with particular
emphasis on the Lie theoretic setting and define certain maximal
parabolic subgroups. Section~\ref{sec:chars} focuses on the ordinary 
characters of the groups introduced. We give constructions for the
characters of the parabolic subgroups, fix notation for the
ordinary unipotent characters and use Deligne-Lusztig theory to
determine several non-unipotent characters of $\SO_5(q)$, $\SO_7(q)$ 
and $\Sp_6(q)$. Section~\ref{sec:decnumbers} treats modular
representations and contains the computation of the decomposition
numbers, mainly dealing with the case $\ell \mid q+1$.  

The tools we use are: Brauer character relations coming from
the $\ell$-regular restrictions of the non-unipotent characters
constructed in Section~\ref{sec:chars}, the decomposition numbers of
$\Sp_4(q)$~\cite{WhiteSp4odd, WhiteSp4even, OkuyamaWaki} and the
approximation of the decomposition matrix of 
$\Sp_6(q)$~\cite{DissKoehler, AnHissSteinberg, WhiteSp6even},
Harish-Chandra induction of projective indecomposable modules (PIMs)
of Levi subgroups, induction of PIMs of cyclic blocks of the
parabolic subgroups introduced in Section~\ref{sec:symporth}  
and a recent method by Dudas~\cite{Dudas} based on Deligne-Lusztig
varieties and their associated complexes of virtual projective
modules. As an intermediate result we obtain the decomposition numbers
of the unipotent characters of $\SO_5(q)$. 

The prime $\ell=2$ is bad for the orthogonal and the symplectic groups
in the sense of~\cite[p.~28]{Carter}. As a consequence the $2$-modular
representation theory of the groups involved seems to be fundamentally 
different and our approach does not lend itself to an easy
adaptation. For example, by~\cite{WhiteSp4oddchar2} the group 
$\Sp_4(q)$ has seven unipotent irreducible $2$-modular Brauer 
characters, but only six ordinary unipotent characters. 
Also, the blocks of the parabolic subgroups we consider are no longer
cyclic for $\ell=2$.

\section{Symplectic and orthogonal groups}
\label{sec:symporth}

In this section we give a description of the symplectic and orthogonal
groups as a pair of dual groups. For the symplectic groups we largely
follow~\cite[Section~2]{AnHissSteinberg}. 

\subsection{Lie theoretic setting}
\label{subsec:lie}
Let $q$ be a power of a prime $p$ and $\F$ an algebraic closure
of the finite field~$\F_q$ with~$q$ elements. We fix a positive
integer $m$ and set $n := 2m+1$ and $n^* := 2m$. Let 
$I_m \in \F^{m \times m}$ be the identity matrix, 
$J_m \in \F^{m \times m}$ the matrix with ones on the anti-diagonal
and zeros elsewhere,
\[
J_n' := \begin{bmatrix}0 & 0 & J_m\\0 & 2 & 0\\J_m & 0 &
  0\end{bmatrix} \quad \text{and} \quad
\tilde{J}_{n^*} := \begin{bmatrix}0 & J_m\\-J_m & 0\end{bmatrix}.
\]
We write vectors $\mbv \in \F^n$ as
$\mbv =\begin{bmatrix}v_m, \dots, v_1, v_0, v_1', \dots,
v_m'\end{bmatrix}^\tr$ and define a quadratic form~$Q_n$ on $\F^n$ by 
\[
Q_n(\mbv) := v_0^2 + \sum_{j=1}^m v_j v_j'.
\]
Let $\{e_m, \dots, e_1,e_0,e_1', \dots, e_m'\}$ be the standard basis
of the vector space $\F^n$. We define 
\begin{eqnarray*}
\G & := & \SO_n(\F) := \{ \mbx \in \SL_n(\F) \mid
Q_n(\mbx\mbv) = Q_n(\mbv) \,\, \text{for all} \,\, \mbv \in
\F^n \},\\
\G^* & := & \Sp_{n^*}(\F) := \{ \mbx \in \GL_{n^*}(\F) \mid \mbx^\tr
\tilde{J}_{n^*} \mbx = \tilde{J}_{n^*} \}.
\end{eqnarray*}
Let $F: \G \to \G$ and $F^*: \G^* \to \G^*$ be the Frobenius maps
raising each matrix entry to its $q$-th power. For 
$t_1, t_2, \dots, t_m \in \F^\times$ we set
\[
h(t_1, \dots, t_m) :=
\mathrm{diag}(t_m^{},\ldots,t_1^{},1,t_1^{-1},\ldots,t_m^{-1})
\in \G,
\]
so that $\T := \{h(t_1, \dots, t_m) \mid t_1, \dots, t_m \in
\F^\times \}$ is an $F$-stable maximally split torus of $\G$.
The root system $\Phi$ of the simple algebraic group $\G$ of
adjoint type with respect to $\T$ is of type $B_m$. We write
$\Phi^+$ for the set of positive roots corresponding to the set 
$\{\alpha_1, \dots, \alpha_m\}$ of simple roots where  
\begin{eqnarray*}
\alpha_1: h(t_1, \dots, t_m) & \mapsto & t_1, \quad \\
\alpha_j: h(t_1, \dots, t_m) & \mapsto & t_j t_{j-1}^{-1} \,\, \text{for}
\,\, j>1. 
\end{eqnarray*}
Note that $\alpha_1$ is a short root, while the other $\alpha_j$ are
long roots. 
The corresponding reflections in the Weyl group $\W$ of $\G$ are  
$w_j := s_j \T$ for $j=1,2,\dots,m$ where 
\begin{eqnarray*}
s_1 & := & \begin{bmatrix} I_{m-1} & 0 & 0 & 0 &
  0\\0&0&0&1&0\\0&0&-1&0&0\\0&1&0&0&0\\0&0&0&0&I_{m-1}\end{bmatrix}
\quad \text{and}\\
s_j & := & \begin{bmatrix}I_{m-j} & 0 & 0 & 0 & 0\\0 & J_2 & 0 & 0 &
  0\\0 & 0 & I_{2j-3} & 0 & 0\\0 & 0 & 0 & J_2 & 0\\0 & 0 & 0 & 0 &
  I_{m-j}\end{bmatrix} \quad \text{for} \ j \in \{2,3,\dots,m\}.
\end{eqnarray*}
In particular, $\W$ is generated by $\{w_1, \dots, w_m\}$ and we
have  
\begin{align} \label{eq:actwj}
\begin{split}
{}^{w_1} h(t_1, \dots, t_m) &= {}^{s_1} h(t_1, \dots, t_m) =
h(t_1^{-1}, t_2, \dots, t_m),\\
{}^{w_j} h(t_1, \dots, t_m) &= {}^{s_j} h(t_1, \dots, t_m) =
h(\dots, t_j, t_{j-1}, \dots) \,\, \text{for} \, j>1.
\end{split}
\end{align}

Similarly, for $t_1, t_2, \dots, t_m \in \F^\times$ we set
\[
h^*(t_1, \dots, t_m) :=
\mathrm{diag}(t_m^{},\ldots,t_1^{},t_1^{-1},\ldots,t_m^{-1})
\in \G^*,
\]
so that $\T^* := \{h^*(t_1, \dots, t_m) \mid t_1, \dots, t_m \in
\F^\times \}$ is an $F^*$-stable maximally split torus of $\G^*$.
The root system $\Phi^*$ of the simple simply-connected algebraic
group~$\G^*$ with respect to $\T^*$ is of type $C_m$. We write
$\Phi^{*+}$ for the set of positive roots corresponding to the set 
$\{\alpha_1^*, \dots, \alpha_m^*\}$ of simple roots where 
\[
\alpha_1^*: h^*(t_1, \dots, t_m) \mapsto t_1^2, \quad 
\alpha_j^*: h^*(t_1, \dots, t_m) \mapsto t_j t_{j-1}^{-1} \,\, \text{for}
\,\, j>1. 
\]
Note that $\alpha_1^*$ is a long root, while the other
$\alpha_j^*$ are short roots. 
The corresponding reflections in the Weyl group $\W^*$ of $\G^*$ are 
$w_j^* := s_j^* \T^*$ for $j=1,2,\dots,m$ where 
\begin{eqnarray*}
s_1^* & := & \begin{bmatrix} I_{m-1} & 0 & 0 &
  0\\0&0&1&0\\0&-1&0&0\\0&0&0&I_{m-1}\end{bmatrix}
\quad \text{and}\\
s_j^* & := & \begin{bmatrix}I_{m-j} & 0 & 0 & 0 & 0\\0 & J_2 & 0 & 0 &
  0\\0 & 0 & I_{2j-4} & 0 & 0\\0 & 0 & 0 & J_2 & 0\\0 & 0 & 0 & 0 &
  I_{m-j}\end{bmatrix} \quad \text{for} \ j \in \{2,3,\dots,m\}.
\end{eqnarray*}
In particular, $\W^*$ is generated by $\{w_1^*, \dots, w_m^*\}$ and we have 
\begin{align} \label{eq:actw*j}
\begin{split}
{}^{w_1^*} h^*(t_1, \dots, t_m) &= {}^{s_1^*} h^*(t_1, \dots, t_m) =
h^*(t_1^{-1}, t_2, \dots, t_m),\\
{}^{w_j^*} h^*(t_1, \dots, t_m) &= {}^{s_j^*} h^*(t_1, \dots, t_m) =
h^*(\dots, t_j, t_{j-1}, \dots) \,\, \text{for} \, j>1.
\end{split}
\end{align}
The pairs $(\G, F)$ and $(\G^*, F^*)$ are in duality in the sense
of~\cite[Chapter~4]{Carter} and there is an isomorphism 
$\delta: \W \to \W^*$ mapping $w_j \mapsto w^*_j$ for
$j=1,2,\dots,m$. 

We are interested in the finite groups 
\[
G := G_n := \SO_n(q) := \G^F \quad \text{and} \quad 
G^* := G^*_{n^*} := \Sp_{n^*}(q) := \G^{*F^*}
\]
with $|G| = |G^*| = q^{m^2}(q^{2m}-1)(q^{2m-2}-1) \cdots (q^2-1)$.
Additionally, we define $\GO^+_{n^*}(q)$, $\GO^-_{n^*}(q)$ 
as in~\cite{HimNoeResSO} and set
$G^*_0 := \Sp_0(q) := \GO^\pm_0(q) := \{1\}$. To apply Deligne-Lusztig
theory we consider $\G^*$ as a subgroup of 
\[
\widetilde{\G}^* := \CSp_{n^*}(\F) := \{ \mbx \in \GL_{n^*}(\F) \mid
\exists \lambda \in \F^\times: \mbx^\tr \tilde{J}_{n^*} \mbx = \lambda
\tilde{J}_{n^*} \}
\]
whose center is connected. We denote the Frobenius map 
$F^*: \widetilde{\G}^* \to \widetilde{\G}^*$ raising each matrix entry
to its $q$-th power also by $F^*$ and define  
$\widetilde{G}^* := \widetilde{G}^*_{n^*} := \widetilde{\G}^{*F^*}$.
In particular, we have $|\widetilde{G}^*| = (q-1) |G^*|$.

\begin{remark} \label{rmk:even_iso_Sp}
For even $q$ there is a natural isomorphism $\SO_n(q) \to \Sp_{n^*}(q)$
mapping each matrix $A$ to the matrix which is obtained from $A$ by
removing the middle row and the middle column.
\end{remark}

\subsection{Parabolic subgroups}
\label{subsec:parabolicsubgrps}

The group $G_n = \SO_n(q)$ acts naturally on the vector space $\F_q^n$
by multiplication from the left. Let $P_n$ be the stabilizer in
$G_n$ of the subspace generated by the basis vector 
$e_m = \begin{bmatrix}1,0,\dots,0\end{bmatrix}^\tr$. Then $P_n$ is a 
maximal parabolic subgroup of $G_n$ with Levi decomposition 
$P_n= L_n \ltimes U_n$, where
\[
L_n  =  \{\mathbf{s}_n(\mathbf{x}, a) \mid \mathbf{x} \in
G_{n-2}, a\in\F_q^\times\}, \quad
U_n  =  \{\mathbf{u}_n(\mathbf{v}) \mid \mathbf{v} \in
\F_q^{n-2}\}
\]
and
\begin{equation}
 \mathbf{s}_n(\mathbf{x}, a) := \begin{bmatrix}a & 0 & 0\\ 0 &
   \mathbf{x} & 0\\ 0 & 0 & a^{-1} \end{bmatrix}, \quad 
\mathbf{u}_n(\mathbf{v}) := \begin{bmatrix} 1 & -\mathbf{v}^\tr
  J_{n-2}' & -Q_{n-2}(\mbv)\\ 0 & 1 & \mathbf{v}\\ 0 & 0 & 1 \end{bmatrix}.
\end{equation}
The Levi subgroup $L_n$ decomposes as 
$L_n = L'_n \times A_n \cong \SO_{n-2}(q) \times \F_q^\times$ where 
\[
L'_n := \{\mathbf{s}_n(\mathbf{x}, 1) \mid \mathbf{x}
\in \SO_{n-2}(q) \} \ \text{and} \ A_n :=
\{\mathbf{s}_n(I_{n-2}, a) \mid a \in \F_q^\times \}. 
\]
The unipotent radical $U_n$ of $P_n$ is elementary
abelian of order $q^{n-2}$. For simplicity, we often write $P$, $L$,
$U$, $A$, $L'$ instead of $P_n$, $L_n$, $U_n$, $A_n$, $L'_n$,
respectively.

Analogously, the group $G^*_{n^*} = \Sp_{n^*}(q)$ acts naturally on
the vector space $\F_q^{n^*}$ by multiplication from the left and we
write $P^*_{n^*}$ for the stabilizer in $G^*_{n^*}$ of the subspace
generated by 
$e_m$. Hence, $P^*_{n^*}$ is a maximal parabolic subgroup with Levi
decomposition $P^*_{n^*} = L^*_{n^*} \ltimes U^*_{n^*}$. The Levi
subgroup $L^*_{n^*}$ is a direct product of a cyclic group 
$A^*_{n^*} \cong \F_q^\times$ and a subgroup 
$L^{*'}_{n^*} \cong \Sp_{n^*-2}(q)$. Details on the groups $A^*_{n^*}$,
$L^{*'}_{n^*}$ and on the unipotent radical $U^*_{n^*}$ are given
in~\cite[2.2]{AnHissSteinberg}. We often write $P^*$, $L^*$, $U^*$,
$A^*$, $L^{*'}$ instead of $P^*_{n^*}$, $L^*_{n^*}$, $U^*_{n^*}$,
$A^*_{n^*}$, $L^{*'}_{n^*}$, respectively. 
The orders of the parabolic subgroups are
\[
|P_n| = |P^*_{n^*}| = q^{m^2}(q-1)(q^{2m-2}-1)(q^{2m-4}-1) \cdots (q^2-1).
\]

\section{Characters}
\label{sec:chars}

In this section we collect some information on the ordinary characters
of the groups $G$, $G^*$, $\widetilde{G}^*$ and some of their parabolic
subgroups. For the symplectic groups we largely
follow \cite{AnHissSteinberg} and \cite{AnHissUnipotent}. 

\subsection{General character theoretic notation}
\label{subsec:genchar}

Let $K$ be a subgroup of a finite group $H$. We write $\Irr(H)$ for
the set of complex irreducible characters of $H$ and $1_H$ for the
trivial character. Let $\langle \cdot,\cdot \rangle_H$ be the usual
scalar product on the space of class functions of $H$. 
If $\chi$ is a character of~$H$ we write 
$\chi\smash\downarrow^H_K$ for the restriction of $\chi$ 
to~$K$. If~$\varphi$ is a character of $K$ we
write~$\varphi\smash\uparrow_K^H$ for the character of $H$ which is 
induced by~$\varphi$. If~$K \unlhd H$ and $\psi$ is a
character of the factor group $H/K$ then we denote its inflation
to~$H$ by $\infl_{H/K}^H(\psi)$. 

\subsection{Characters of parabolic subgroups}
\label{subsec:charsparabolic}

We give a brief summary of the construction of the ordinary
irreducible characters of the parabolic subgroup $P$ of $G$ in~\cite{HimNoeResSO}.
For $n \ge 5$ the conjugation action
of $P$ on $\Irr(U)$ has four orbits and we choose a set of
representatives $\{1_U, \lambda^0, \lambda^+, \lambda^-\}$ as
in~\cite{HimNoeResSO}. The characters $\lambda^0$, $\lambda^+$,
$\lambda^-$ are non-trivial linear characters and we say that 
$\chi \in \Irr(P)$ is of Type~1 if it covers~$1_U$ and it is of 
Type~$\varepsilon$ if it covers $\lambda^\varepsilon$ for 
$\varepsilon \in \{0,+,-\}$. The irreducible characters of Type~$1$ 
are obtained from $\Irr(L)$ via inflation, we set 
${}^1\psi_\sigma := \infl_L^P(\sigma)$ for $\sigma \in \Irr(L)$.
The inertia subgroup in $P$ of the linear character~$\lambda^0$ is 
$I^0 = \tilde{P}_{n-2} \ltimes U$ where $\tilde{P}_{n-2} \cong P_{n-2}$.
We identify $\Irr(\tilde{P}_{n-2})$ with $\Irr(P_{n-2})$ via the
bijection in~\cite[3.3]{HimNoeResSO} and extend $\lambda^0$ 
trivially to $\hat{\lambda}^0 \in \Irr(I^0)$. The irreducible
characters of Type~$0$ are parameterized by $\Irr(P_{n-2})$ via 
\[
{}^0\psi_\mu := (\hat{\lambda}^0 \cdot
\infl_{\tilde{P}_{n-2}}^{I^0}(\mu))\smash\uparrow_{I^0}^P \quad
\text{for} \quad \mu \in \Irr(P_{n-2}) = \Irr(\tilde{P}_{n-2}).
\]
We say that $\chi \in \Irr(P)$ is of Type~$0.\varepsilon$ if 
$\chi = {}^0\psi_\mu$ for some $\mu \in \Irr(P_{n-2})$ of
Type~$\varepsilon$. The inertia subgroup of $\lambda^\pm$ in 
$P$ is $I^\pm = L^\pm_n \ltimes U$ where $L^\pm_n \cong \GO^\pm_{n-3}(q)$
and we extend $\lambda^\pm$ trivially to $\hat{\lambda}^\pm \in \Irr(I^\pm)$. 
The irreducible characters of Type~$\pm$ are parameterized by
$\Irr(\GO^\pm_{n-3}(q)) = \Irr(L^\pm_n)$ via  
\[
{}^\pm\psi_\vartheta := (\hat{\lambda}^\pm \cdot
\infl_{L^\pm_n}^{I^\pm}(\vartheta))\smash\uparrow_{I^\pm}^P
\quad \text{for} \quad \vartheta \in \Irr(\GO^\pm_{n-3}(q)) .
\]
For $n=3$ the orbits of $\lambda^0$ and $\lambda^-$ do not exist and
there is just one irreducible character of Type~$+$ of $P_3$,
namely the one corresponding to the trivial character of 
$\GO_0^+(q) := \{1\}$. 

Suppose that $q$ is odd. The ordinary irreducible characters of the
parabolic subgroup $P^*$ of~$G^*$ were determined by An and Hiss in
\cite{AnHissSteinberg}. We choose a set of representatives 
$\{1_{U^*}, \lambda, \rho_1, \rho_2\}$ for the four orbits of 
$P^*$ on $\Irr(U^*)$ as in \cite{AnHissSteinberg}; see also 
\cite[Sections~4, 5]{AnHissUnipotent} for the particular choice of
$\rho_1$, $\rho_2$. The character~$\lambda$ is a non-trivial linear
character while $\rho_1(1) = \rho_2(1) = q^{m-1}$. Following
\cite{AnHissSteinberg}, we say that $\chi \in \Irr(P^*)$ is of
Type~$1$ if it covers~$1_{U^*}$, it is of Type~$2$ if it covers~$\lambda$, and 
it is of Type~$3$ if it covers $\rho_1$ or $\rho_2$. The irreducible
characters of Type~$1$ are obtained from $\Irr(L^*)$ via inflation,
and we set ${}^1\psi_\sigma := \infl_{L^*}^{P^*}(\sigma)$ 
for $\sigma \in \Irr(L^*)$. The inertia subgroup of~$\lambda$ in $P^*$
is $I^2 = \tilde{P}^*_{n^*-2} \ltimes U^*$ where 
$\tilde{P}^*_{n^*-2} \cong P^*_{n^*-2}$. We identify 
$\Irr(\tilde{P}^*_{n^*-2})$ with $\Irr(P^*_{n^*-2})$ via the
bijection in~\cite[Section~2]{AnHissUnipotent} and extend the linear
character $\lambda$ trivially to $\hat{\lambda} \in \Irr(I^2)$. The
irreducible characters of Type~$2$ are parameterized by
$\Irr(P^*_{n^*-2})$ via  
\[
{}^2\psi_\mu := (\hat{\lambda} \cdot
\infl_{\tilde{P}^*_{n^*-2}}^{I^2}(\mu))\smash\uparrow_{I^2}^{P^*} \quad
\text{for} \quad \mu \in \Irr(P^*_{n^*-2}) = \Irr(\tilde{P}^*_{n^*-2}).
\]
We say that $\chi \in \Irr(P^*)$ is of Type~$2.\varepsilon$ if 
$\chi = {}^2\psi_\mu$ for some $\mu \in \Irr(P^*_{n^*-2})$ of
Type~$\varepsilon$. The characters $\rho_1$ and $\rho_2$ have the same
inertia subgroup $I^3 = (L^{*'} \times Z^*) \ltimes U^*$  
where $Z^* := \langle -I_{n^*} \rangle$. For $i=1,2$ we choose an
extension $\hat{\rho_i}$ of the character $\rho_i$ to~$I^3$ as
in~\cite[2.3.3]{AnHissSteinberg}. For 
$\vartheta \in \Irr(L^{*'}) = \Irr(\Sp_{n^*-2}(q))$ we write 
$\vartheta \cdot 1_{Z^*}^+$ for the trivial extension of $\vartheta$ to
$L^{*'} \times Z^*$ and $\vartheta \cdot 1_{Z^*}^-$ for the non-trivial
extension. The irreducible characters of Type~$3$ are parameterized by  
$\Irr(\Sp_{n^*-2}(q))$ via 
\[
{}^3\psi^{i,\varepsilon}_\vartheta := (\hat{\rho}_i \cdot
\infl_{L^{*'}Z^*}^{I^3}(\vartheta \cdot
1_{Z^*}^\varepsilon))\smash\uparrow_{I^3}^{P^*} \,\,
\text{where} \,\, \vartheta \in \Irr(\Sp_{n^*-2}(q)), \, i=1,2, \, 
\varepsilon \in \{+,-\}.
\]

It is a slight abuse of notation to denote some of the inertia
subgroups and irreducible characters of~$P$ and $P^*$ by the same
symbols (the characters of Type~$1$, for example). However, we will
always make clear whether we are working with the orthogonal or the
symplectic groups, so that there should be no confusion.

\subsection{Characters of symplectic and orthogonal groups}
\label{subsec:charsSpSO}

In this section we fix notation for the ordinary unipotent
irreducible characters of the groups $\SO_{2m+1}(q)$, $\Sp_{2m}(q)$
and $\CSp_{2m}(q)$ for $m=2,3$. We also construct some non-unipotent
ordinary characters of these groups as linear combinations of
Deligne-Lusztig characters. When we speak of unipotent characters we 
always mean irreducible characters. 

Let $q$ be an arbitrary prime power. Each unipotent character of the
groups $\SO_{2m+1}(q)$, $\Sp_{2m}(q)$ and $\CSp_{2m}(q)$ is labeled by
a symbol $\Lambda$ or a triple $[\alpha,\beta,d]$ where
$(\alpha,\beta)$ is a bipartition and $d$ is the defect of $\Lambda$;
see \cite[Section~13.8]{Carter} and \cite[Section~7]{HimNoeResSO}. 
Each of the groups $\SO_5(q)$, $\Sp_4(q)$ and $\CSp_4(q)$ has six
and each of the groups $\SO_7(q)$, $\Sp_6(q)$ and $\CSp_6(q)$ has
twelve unipotent characters. Their degrees and the symbols and
bipartitions labeling these characters are given in 
Tables~\ref{tab:labelsSO5} and~\ref{tab:labelsSO7}. In these tables we
use the abbreviations $\phi_1 := q-1$, $\phi_2 := q+1$, 
$\phi_3 := q^2+q+1$, $\phi_4 := q^2+1$ and $\phi_6 :=
q^2-q+1$. Often, we identify a unipotent character with its label
$[\alpha,\beta,d]$.

The character tables of $\Sp_4(q)$ and $\CSp_4(q)$ were computed in
\cite{Enomoto}, \cite{ShinodaCSp4}, \cite{SrinivasanSp4} and
\cite{YamadaSp4}. The character table of $\CSp_6(q)$ for odd $q$ and
$\Sp_6(q)$ for even $q$ were determined in~\cite{DissLuebeck}; see
also the \CHEVIE\ library \cite{CHEVIE}.

\begin{table}[h]
 \centering
 \begin{tabular}{|c|c|c||c|c|c|}
  \hline
  Bipartition & Symbol & Degree &  Bipartition & Symbol & Degree \\
  \hline
  ${[2,-,1]}$ & $\begin{pmatrix}2\\-\end{pmatrix}$ & $1$ &
    ${[1,1,1]}$ & $\begin{pmatrix}0 \: 2\\1\end{pmatrix}$ &
     $\frac{1}{2}q\phi_2^2$\\
    ${[-,-,3]}$ & $\begin{pmatrix} 0 \: 1\: 2\\ - \end{pmatrix}$ &
     $\frac{1}{2}q\phi_1^2$ &
   ${[-,2,1]}$ & $\begin{pmatrix} 0\: 1\\2 \end{pmatrix}$ &
    $\frac{1}{2}q\phi_4$\\
   ${[1^2,-,1]}$ & $\begin{pmatrix}1\: 2\\ 0\end{pmatrix}$ & 
    $\frac{1}{2}q\phi_4$ &
  ${[-,1^2,1]}$ & $\begin{pmatrix} 0\: 1\: 2\\ 1\:2 \end{pmatrix}$ & $q^4$\\
  \hline
 \end{tabular}
 \caption{Labels and degrees of the unipotent characters of
   $\SO_5(q)$, $\Sp_4(q)$, $\CSp_4(q)$.}
 \label{tab:labelsSO5}
\end{table}

\begin{table}
 \centering
 \begin{tabular}{|c|c|c||c|c|c|}
  \hline
  Bipartition & Symbol & Degree &  Bipartition & Symbol & Degree \\
  \hline
  ${[3,-,1]}$ & $\begin{pmatrix}3\\-\end{pmatrix}$ & $1$ &
    ${[1^2,1,1]}$ & $\begin{pmatrix}1\:2\\1\end{pmatrix}$ & 
     $q^3\phi_3\phi_6$\\
  ${[2,1,1]}$ & $\begin{pmatrix}0\:3\\1\end{pmatrix}$ &
      $\frac{1}{2}q\phi_3\phi_4$ & ${[1,1^2,1]}$ &
      $\begin{pmatrix}0\:1\:3\\1\:2\end{pmatrix}$ & 
      $\frac{1}{2}q^4\phi_3\phi_4$\\ 
  ${[-,3,1]}$ & $\begin{pmatrix}0\:1\\3\end{pmatrix}$ &
      $\frac{1}{2}q\phi_4\phi_6$ & ${[-,2\,1,1]}$ &
      $\begin{pmatrix}0\:1\:2\\1\:3\end{pmatrix}$ & 
      $\frac{1}{2}q^4\phi_2^2\phi_6$\\
  ${[2\,1,-,1]}$ & $\begin{pmatrix}1\:3\\0\end{pmatrix}$ &
      $\frac{1}{2}q\phi_2^2\phi_6$ & ${[1^3,-,1]}$ &
      $\begin{pmatrix}1\:2\:3\\0\:1\end{pmatrix}$ & 
      $\frac{1}{2}q^4\phi_4\phi_6$\\
  ${[1,-,3]}$ & $\begin{pmatrix}0\:1\:3\\-\end{pmatrix}$ &
      $\frac{1}{2}q\phi_1^2\phi_3$ & ${[-,1,3]}$ &
      $\begin{pmatrix}0\:1\:2\:3\\1\end{pmatrix}$ & 
      $\frac{1}{2}q^4\phi_1^2\phi_3$\\
  ${[1,2,1]}$ & $\begin{pmatrix}0\:2\\2\end{pmatrix}$ &
      $q^2\phi_3\phi_6$ & ${[-,1^3,1]}$ &
      $\begin{pmatrix}0\:1\:2\:3\\1\:2\:3\end{pmatrix}$ & $q^9$\\
  \hline
 \end{tabular}
 \caption{Labels and degrees of the unipotent characters of
   $\SO_7(q)$, $\Sp_6(q)$, $\CSp_6(q)$.}
 \label{tab:labelsSO7}
\end{table}

In the following we construct some non-unipotent irreducible
characters of the groups $\SO_5(q)$ and $\SO_7(q)$ for odd $q$ as
linear combinations of Deligne-Lusztig characters. These linear
combinations will be used in Section~\ref{sec:decnumbers} to derive
relations between Brauer characters leading to lower bounds for
decomposition numbers. In the remaining part of this section we always
assume that $q$ is odd. 

Suppose that $m=2$ so that $\G = \SO_5(\F)$ and $\G^* = \Sp_4(\F)$. 
Since $(\G, F)$ and $(\G^*, F^*)$ are in duality and the center
of $\G$ is connected, the Lusztig series of the irreducible characters
of $G = \G^F = \SO_5(q)$ are parameterized by the $F^*$-stable
semisimple conjugacy classes of $\G^*$ or equivalently by the 
semisimple conjugacy classes of $G^* = \G^{*F^*} = \Sp_4(q)$. To
construct some of these irreducible characters we proceed along the
lines of \cite[Section~4.1 and Section~7]{DissLuebeck}. 
Recall that the action of the Frobenius map $F^*$ on
the root system $\Phi^*$ of $\G^*$ is trivial. 

For $t_1, t_2 \in \F^\times$ with $t_1 = \pm 1$ and $t_2^{q+1}=1$,
$t_2 \neq \pm 1$, we set 
\[ 
g(t_1, t_2) := h^*(t_1,t_2), \quad \Pi^* := \{\alpha_1^*\}, \quad
w_{212} := w_2 w_1 w_2, \quad w_{212}^* := \delta(w_{212})^{-1},
\]
where $\delta: \W \to \W^*$ is the isomorphism mentioned at the end of
Section~\ref{subsec:lie}. Furthermore, we write $\Psi^*$ for the
closed subset of $\Phi^*$ generated by $\Pi^*$. Using
\eqref{eq:actw*j} it is straightforward to see that 
\begin{itemize}
\item[(i)] ${}^{w_{212}^*}F^*(g(t_1,t_2)) = g(t_1,t_2)$,
\item[(ii)] $\alpha_1^*(g(t_1,t_2)) = 1$, and
\item[(iii)] $\alpha^*(g(t_1,t_2)) \neq 1$ for all 
  $\alpha^* \in \Phi^{*+} \setminus \Psi^*$.
\end{itemize}
It follows from (i)~that the element $g(t_1,t_2)$ is conjugate in
$\G^*$ to some semisimple element of $G^* = \G^{*F^*} \subseteq \G^*$. More
specifically: By the surjectivity of the Lang map there is 
$x \in \G^*$ such that $x^{-1} F^*(x) \in w_{212}^*$ and for this 
$x$ property~(i)~implies that ${}^x g(t_1,t_2) \in G^*$. The
centralizer $C_{\G^*}(g(t_1,t_2))$ is a connected reductive group and
properties (ii)~and (iii)~imply that the root system $\Psi^*$ of 
$C_{\G^*}(g(t_1,t_2))$ with respect to the maximal torus $\T^*$ has type
$\tilde{A}_1$ and that the Weyl group $\W_C^*$ of $C_{\G^*}(g(t_1,t_2))$ is
generated by $w^*_1$. 

Since ${}^x g(t_1,t_2)$ is $F^*$-stable its centralizer 
$C_{\G^*}({}^x g(t_1,t_2))$ is also $F^*$-stable. This induces an
action of $F^*$ on $C_{\G^*}(g(t_1,t_2))$,
on~$\Psi^*$ and on $\W_C^*$ via the isomorphism 
$C_{\G^*}(g(t_1,t_2)) \to C_{\G^*}({}^x g(t_1,t_2))$, 
$y \mapsto {}^xy$. The action of $F^*$ on
$\Psi^*$ and on~$\W_C^*$ is trivial and the Dynkin type
of $C_{\G^*}({}^x g(t_1,t_2))^{F^*}$ is $\tilde{A}_1$. The 
$F^*$-conjugacy classes coincide with the conjugacy
classes of $\W_C^*$ and are given by: $C_1 = \{1\}$, $C_2 = \{w^*_1\}$.
We apply \cite[Lemma~7.1]{DissLuebeck} to construct the irreducible
characters of $\SO_5(q)$ in the Lusztig series labeled by $g(t_1,t_2)$. 
The group $\W_C^*$ is isomorphic to the symmetric group $S_2$ and
its character table is 
\begin{center}
\begin{tabular}{|c|rr|}
\hline
& $C_1$ & $C_2$ \\\hline
$\phi_1$ & $1$ & $1$ \\
$\phi_2$ & $1$ & $-1$\\\hline
\end{tabular}
\end{center}
We use the notation from~\cite[Lemma~7.1]{DissLuebeck}. Since the
action of the Frobenius map $F^*$ on~$\W_C^*$ is trivial we have 
$\hat{\W}_C^{*F^*} = \Irr(\W_C^*)$ and $\alpha_\phi = 1$ for all 
$\phi \in \Irr(\W_C^*)$. For $w^* \in \W_C^*$ we write~$R_{w^*}$
for the Deligne-Lusztig character of $C_{G^*}(g(t_1,t_2))^{F^*}$
corresponding to the trivial character of a maximal torus obtained
from $\T^*$ by twisting with $w^*$. By~\cite[Lemma~7.1]{DissLuebeck} the
unipotent irreducible characters of $C_{G^*}(g(t_1,t_2))^{F^*}$ are
given by 
$R_{\phi_1} = \varepsilon_1 \frac12 ( R_1 + R_{w^*_1} )$ and 
$R_{\phi_2} = \varepsilon_2 \frac12 ( R_1 - R_{w^*_1} )$,
where $\varepsilon_1$, $\varepsilon_2$ are complex roots of
unity. Applying the Jordan decomposition of characters we obtain 
the irreducible characters 
\begin{align} \label{eq:irr12SO5}
\begin{split}
\chi_1(t_1,t_2) &= \frac{\varepsilon'_1}{2} ( R_{w_{212}}(t_1,t_2) + 
R_{w_{212} w_1}(t_1,t_2) ), \\
\chi_2(t_1,t_2) &= \frac{\varepsilon'_2}{2} ( R_{w_{212}}(t_1,t_2) - 
R_{w_{212} w_1}(t_1,t_2) ),
\end{split}
\end{align}
of $G = \SO_5(q)$, where $\varepsilon'_1$, $\varepsilon'_2$ are
complex roots of unity. In~\eqref{eq:irr12SO5} we write $R_{w}(t_1,t_2)$
for the Deligne-Lusztig character of $G$ which is labeled by the
$F$-stable maximal torus $\T_{w}$ of $\G$ which is obtained
from $\T$ by twisting with $w \in \W$ and the linear character of
$\T_{w}^F$ corresponding to a fixed conjugate of $g(t_1,t_2)$ in
$\T_{\delta(w)^{-1}}^{*F^*}$. Since $R_{w}(t_1,t_2)(1) = R_{w}(1)(1)$ 
and since we can easily evaluate the Deligne-Lusztig characters
$R_{w}(1)$ at the identity element $1$ with \CHEVIE\ we get
$\varepsilon_1' = \varepsilon_2' = -1$ and the degrees are 
$\chi_1(t_1,t_2)(1) = 1 \cdot (q-1)(q^2+1)$, 
$\chi_2(t_1,t_2)(1) = q \cdot (q-1)(q^2+1)$.

For $t_1, t_2 \in \F^\times$ with $t_1^{q+1} = t_2^{q+1} = 1$,
$t_1,t_2 \neq \pm 1$ and $t_2 \neq t_1^{\pm1}$, we set 
\[ 
g_\reg(t_1, t_2) := h^*(t_1,t_2), \quad w_0 := w_1 w_2 w_1 w_2, \quad
w^*_0 := \delta(w_0)^{-1} = w^*_2 w^*_1 w^*_2 w^*_1,
\]
so that $w_0$ and $w^*_0$ are the longest elements of $\W$, $\W^*$,
respectively. Using \eqref{eq:actw*j} it is straightforward to see that 
${}^{w^*_0}F^*(g_\reg(t_1,t_2)) = g_\reg(t_1,t_2)$ and
$\alpha^*(g_\reg(t_1,t_2)) \neq 1$ for all $\alpha^* \in \Phi^{*+}$.
Thus, $g_\reg(t_1,t_2)$ is conjugate in $\G^*$ to some regular
semisimple element in $\T_{w^*_0}^{*F^*}$ where $\T_{w^*_0}$ is an
$F^*$-stable maximal torus of $\G^*$ which is obtained from $\T^*$ by
twisting with $w^*_0$. As described in~\cite[Section~3]{DissLuebeck},
the element $g_\reg(t_1,t_2)$ corresponds to a linear character of
$\T_{w_0}^F$ in general position. Hence, the corresponding
Deligne-Lusztig character $R_{w_0}(t_1,t_2)$ of $G=\G^F$ is
irreducible up to sign. From \CHEVIE\ we get that
$R_{w_0}(t_1,t_2)(1) = (q-1)^2(q^2+1)$, thus
\begin{equation} \label{eq:irrregSO5}
\chi_3(t_1,t_2) := R_{w_0}(t_1,t_2) \in \Irr(G).
\end{equation}

Now suppose that $m=3$ so that $\G = \SO_7(\F)$ and $\G^* = \Sp_6(\F)$. 
We still assume that $q$ is odd. Similar to the case $m=2$ we use
the methods from~\cite[Section~4.1 and Section~7]{DissLuebeck} to
construct irreducible characters of $G = \G^F = \SO_7(q)$.
We define elements of the Weyl group $\W$ by:
\[
w_9 := w_1 w_2 w_1 w_3 w_2 w_1, \,\,
w_{13} := w_2 w_1 w_3 w_2 w_1 w_3 w_2, \,\,
w_{23} := w_2 w_1 w_3 w_2 w_1 w_3 w_2 w_3.
\]
Furthermore, let $w_{32}$ be the longest element of $\W$ and 
set $w_j^* := \delta(w_j)^{-1} \in \W^*$ for $j=9,13,23,32$.
Additionally, we define semisimple elements of $\G^*$ by
\begin{enumerate}
\item[1)] $g_9(t) := h^*(t,t,t)$ \,\, for $t \in \F^\times$ with $t^{q+1}=1$,
  $t \neq \pm 1$, 

\item[2)] $g_{13}(t_1,t_2) := h^*(t_1,t_2,t_2)$ \,\, for $t_1, t_2 \in \F^\times$
  with $t_1 = \pm 1$, $t_2^{q+1}=1$, $t_2 \neq \pm 1$,

\item[3)] $g_{23}(t_1,t_2,t_3) := h^*(t_1,t_2,t_3)$ \,\, for $t_1, t_2, t_3 \in \F^\times$
  with $t_1 = \pm 1$, $t_2^{q+1}=t_3^{q+1}=1$, $t_2 \neq \pm 1$, 
  $t_3 \neq \pm 1, t_2^{\pm 1}$,

\item[4)] $g_{32}(t_1,t_2,t_3) := h^*(t_1,t_2,t_3)$ \,\, for $t_1, t_2, t_3 \in \F^\times$
  with $t_1^{q+1}=t_2^{q+1}=t_3^{q+1}=1$, $t_i \neq t_j^{\pm1}$ for
  all $i,j=1,2,3$,
\end{enumerate}
where the notation is motivated by the analogy with~\cite[Table~17]{DissLuebeck}. 
For each element $g_9(t)$, where $t$ satisfies the conditions in 1),
we construct characters $\chi_{9,1}(t), \chi_{9,2}(t), \chi_{9,3}(t)
\in \Irr(G)$ which can be written as rational linear combinations of
Deligne-Lusztig characters as follows:
\begin{eqnarray} \label{eq:chi9_SO7}
\chi_{9,1}(t) \hspace{-0.2cm} & = & \hspace{-0.2cm} \frac{1}{6} ( 
  R_{w_9 w_2 w_3 w_2} (t) + 3 R_{w_9}(t) + 2 R_{w_9 w_2}(t) ), \nonumber\\ 
\chi_{9,2}(t) \hspace{-0.2cm} & = & \hspace{-0.2cm}
\frac{1}{6} ( 2 R_{w_9 w_2 w_3 w_2}(t) - 2 R_{w_9 w_2}(t) ),\\ 
\chi_{9,3}(t) \hspace{-0.2cm} & = & \hspace{-0.2cm}
-\frac{1}{6} ( R_{w_9 w_2 w_3 w_2}(t) - 3 R_{w_9}(t) + 2 R_{w_9 w_2}(t) ). \nonumber
\end{eqnarray}
By $R_{w}(t)$ we mean the Deligne-Lusztig character of the group 
$G = \SO_7(q)$ which is labeled by the $F$-stable maximal torus $\T_{w}$
of $\G$ which is obtained from $\T$ by twisting with $w \in \W$ and
the linear character of $\T_{w}^F$ corresponding to a fixed conjugate
of $g_9(t)$ in $\T_{\delta(w)^{-1}}^{*F^*}$. 
The degrees of $\chi_{9,1}(t)$, $\chi_{9,2}(t)$, $\chi_{9,3}(t)$ are
$\phi_1^2\phi_3\phi_4$, $q\phi_1^3\phi_3\phi_4$, $q^3\phi_1^2\phi_3\phi_4$,
respectively. 

Similarly, for each element $g_{13}(t_1,t_2)$, where $t_1$, $t_2$
satisfy the conditions in~2), there are characters
$\chi_{13,1}(t_1,t_2), \chi_{13,2}(t_1,t_2), \chi_{13,3}(t_1,t_2),
\chi_{13,4}(t_1,t_2) \in \Irr(G)$ which can be written as 
linear combinations of Deligne-Lusztig characters as follows: 
\begin{align} \label{eq:chi13_SO7}
\begin{split}
\chi_{13,1}( \hspace{-0.05cm} t_1, t_2 \hspace{-0.05cm} ) &= 
-\frac{1}{4} ( R_{w_{13}}( \hspace{-0.05cm} t_1, t_2 \hspace{-0.05cm}
) \hspace{-0.05cm} + \hspace{-0.05cm}  R_{w_{13}
  w_1}( \hspace{-0.05cm} t_1, t_2 \hspace{-0.05cm} ) \hspace{-0.05cm}
+ \hspace{-0.05cm}  R_{w_{13} w_3}( \hspace{-0.05cm} t_1,
t_2 \hspace{-0.05cm} )  
\hspace{-0.05cm} + \hspace{-0.05cm}  R_{w_{13} w_1 w_3}( \hspace{-0.05cm} t_1, t_2 \hspace{-0.05cm} ) ),\\
\chi_{13,2}( \hspace{-0.05cm} t_1, t_2 \hspace{-0.05cm}
) &=  
-\frac{1}{4} ( R_{w_{13}}( \hspace{-0.05cm} t_1, t_2 \hspace{-0.05cm}
) \hspace{-0.05cm} + \hspace{-0.05cm}  R_{w_{13}
  w_1}( \hspace{-0.05cm} t_1, t_2 \hspace{-0.05cm} ) \hspace{-0.05cm}
- \hspace{-0.05cm} R_{w_{13} w_3}( \hspace{-0.05cm} t_1,
t_2 \hspace{-0.05cm} ) \hspace{-0.05cm} - \hspace{-0.05cm} R_{w_{13}
  w_1 w_3}( \hspace{-0.05cm} t_1, t_2 \hspace{-0.05cm} ) ),\\ 
\chi_{13,3}( \hspace{-0.05cm} t_1, t_2 \hspace{-0.05cm} ) &= 
-\frac{1}{4} ( R_{w_{13}}( \hspace{-0.05cm} t_1, t_2 \hspace{-0.05cm}
) \hspace{-0.05cm} - \hspace{-0.05cm} R_{w_{13} w_1}( \hspace{-0.05cm}
t_1, t_2 \hspace{-0.05cm} ) \hspace{-0.05cm} + \hspace{-0.05cm}
R_{w_{13} w_3}( \hspace{-0.05cm} t_1, t_2 \hspace{-0.05cm} )  
\hspace{-0.05cm}- \hspace{-0.05cm}R_{w_{13} w_1 w_3}( \hspace{-0.05cm} t_1, t_2 \hspace{-0.05cm} ) ),\\
\chi_{13,4}( \hspace{-0.05cm} t_1, t_2 \hspace{-0.05cm} ) &= 
-\frac{1}{4} ( R_{w_{13}}( \hspace{-0.05cm} t_1, t_2 \hspace{-0.05cm}
) \hspace{-0.05cm} - \hspace{-0.05cm} R_{w_{13} w_1}( \hspace{-0.05cm}
t_1, t_2 \hspace{-0.05cm} ) \hspace{-0.05cm} - \hspace{-0.05cm}
R_{w_{13} w_3}( \hspace{-0.05cm} t_1, t_2 \hspace{-0.05cm} )  
\hspace{-0.05cm} + \hspace{-0.05cm}  R_{w_{13} w_1 w_3}( \hspace{-0.05cm} t_1, t_2 \hspace{-0.05cm} ) ).
\end{split}
\end{align}
The degrees of the characters $\chi_{13,1}(t_1,t_2)$, $\chi_{13,2}(t_1,t_2)$,
$\chi_{13,3}(t_1,t_2)$, $\chi_{13,4}(t_1,t_2)$ are
$\phi_1\phi_3\phi_4\phi_6$, $q\phi_1\phi_3\phi_4\phi_6$, 
$q\phi_1\phi_3\phi_4\phi_6$, $q^2\phi_1\phi_3\phi_4\phi_6$, respectively.

Furthermore, for each element $g_{23}(t_1,t_2,t_3)$, where $t_1$, $t_2$,
$t_3$ satisfy the conditions in~3), there are characters
$\chi_{23,1}(t_1,t_2,t_3), \chi_{23,2}(t_1,t_2,t_3) \in \Irr(G)$ which
can be written as rational linear combinations of Deligne-Lusztig
characters as follows: 
\begin{align} \label{eq:chi23_SO7}
\begin{split}
\chi_{23,1}(t_1,t_2,t_3) &= \frac{1}{2} ( R_{w_{23}}(t_1,t_2,t_3) + 
R_{w_{23} w_1}(t_1,t_2,t_3) ),\\
\chi_{23,2}(t_1,t_2,t_3) &= \frac{1}{2} ( R_{w_{23}}(t_1,t_2,t_3) -
R_{w_{23} w_1}(t_1,t_2,t_3) ).
\end{split}
\end{align}
The degrees of $\chi_{23,1}(t_1,t_2,t_3)$,
$\chi_{23,2}(t_1,t_2,t_3)$ are
$\phi_1^2\phi_3\phi_4\phi_6$, $q\phi_1^2\phi_3\phi_4\phi_6$,
respectively. 

Finally, for each element $g_{32}(t_1,t_2,t_3)$, where $t_1$, $t_2$,
$t_3$ satisfy the conditions in~4), the character
\begin{equation} \label{eq:chi32_SO7}
\chi_{32}(t_1,t_2,t_3) := -R_{w_{32}}(t_1,t_2,t_3)
\end{equation}
is an irreducible character of $G$ of degree $\phi_1^3\phi_3\phi_4\phi_6$, where
$R_{w_{32}}(t_1,t_2,t_3)$ is the Deligne-Lusztig character
corresponding to the maximal torus obtained from $\T$ by twisting with
the longest element $w_{32}$ and a linear character in general position. 

We only sketch the construction of $\chi_{9,1}(t)$, $\chi_{9,2}(t)$ and
$\chi_{9,3}(t)$. The construction of the remaining characters is
similar and easier. Let $t \in \F^\times$ with $t^{q+1}=1$, 
$t \neq \pm 1$. We set $\Pi^*_9 := \{\alpha^*_2, \alpha^*_3\}$
and $w^*_9 := \delta(w_9)^{-1} = w_1^* w_2^* w_3^* w_1^* w_2^* w_1^*$
and we write $\Psi^*_9$ for the closed subset of $\Phi^*$ generated by
$\Pi^*_9$. Using \eqref{eq:actw*j} it is easy to see that
\begin{itemize}
\item[(i)] ${}^{w^*_9}F^*(g_9(t)) = g_9(t)$,
\item[(ii)] $\alpha^*(g_9(t)) = 1$ for all $\alpha^* \in \Pi^*_9$, and
\item[(iii)] $\alpha^*(g_9(t)) \neq 1$ for all $\alpha^* \in \Phi^{*+} \setminus \Psi^*_9$.
\end{itemize}
As for $m=2$, it follows from (i)~that the element $g_9(t)$ is
conjugate in $\G^*$ to some element of 
$G^* = \G^{*F^*} \subseteq \G^*$. More specifically: By the
surjectivity of the Lang map there is $x \in \G^*$ such that 
$x^{-1} F^*(x) \in w^*_9$ and for this $x$ property~(i)~implies that 
${}^x g_9(t) \in G^*$. The centralizer $C_{\G^*}(g_9(t))$
is a connected reductive group and (ii), (iii)~imply
that the root system~$\Psi^*_9$ of $C_{\G^*}(g_9(t))$ with respect 
to $\T^*$ has type~$A_2$ and that the Weyl group
$\W^*_9$ of $C_{\G^*}(g_9(t))$ is generated by $\{w^*_2, w^*_3\}$. 
The action of $F^*$ on $\Psi^*_9$ is given by 
$\alpha_2^* \mapsto \alpha_3^*$ and $\alpha_3^* \mapsto
\alpha_2^*$ and the action of $F^*$ on~$\W^*_9$ is given by
$w^*_2 \mapsto w^*_3$ and $w^*_3 \mapsto w^*_2$. The Dynkin
type of $C_{\G^*}({}^x g_9(t))^{F^*}$ is ${{}^2A}_2$ and the
$F^*$-conjugacy classes of $\W^*_9$ are: 
\[
\bar{C}_1 := \{1, w^*_2 w^*_3, w^*_3 w^*_2\}, \quad \bar{C}_2 := \{w^*_2 w^*_3
w^*_2\} \quad \text{and} \quad \bar{C}_3 := \{w^*_2, w^*_3\}.
\]
The group $\W^*_9$ is isomorphic to the symmetric group $S_3$ and its
conjugacy classes are $C_1 = \{1\}$, $C_2 = \{w^*_2, w^*_3, w^*_2 w^*_3 w^*_2\}$
and $C_3 = \{w^*_2 w^*_3, w^*_3 w^*_2\}$. The character table of $\W^*_9$ is
\begin{center}
\begin{tabular}{|c|rrr|}
\hline
& $C_1$ & $C_2$ & $C_3$\\\hline
$\phi_1$ & $1$ & $1$ & $1$\\
$\phi_2$ & $2$ & $0$ & $-1$\\
$\phi_3$ & $1$ & $-1$ & $1$\\\hline
\end{tabular}
\end{center}
The Frobenius map $F^*$ acts on $\W^*_9$ in the same way as the inner
automorphism induced by the element $w^*_0 := w^*_2 w^*_3 w^*_2$. Thus, in the
notation of \cite[Lemma~7.1]{DissLuebeck} we have 
$\hat{\W}_9^{*F^*} = \Irr(\W^*_9)$ and $\alpha_\phi = w^*_0$ for all 
$\phi \in \Irr(\W_9^*)$. For $w^* \in \W^*_9$ we write~$R_{w^*}$ for the
Deligne-Lusztig character of $C_{G^*}(g_9(t))^{F^*}$ corresponding to
the trivial character of a maximal torus obtained from $\T^*$ by
twisting with $w^*$. By~\cite[Lemma~7.1]{DissLuebeck} the unipotent
irreducible characters of $C_{G^*}(g_9(t))^{F^*}$ are 
\begin{eqnarray*}
R_{\phi_1} & = & \varepsilon_1 \frac16 ( R_{w^*_0} +3 R_1 + 2
R_{w^*_2} ), \quad
R_{\phi_2} = \varepsilon_2 \frac16 ( 2 R_{w^*_0} - 2 R_{w^*_2} ),\\
R_{\phi_3} & = & \varepsilon_3 \frac16 ( R_{w^*_0} - 3 R_1 + 2 R_{w^*_2} ),
\end{eqnarray*}
where $\varepsilon_1$, $\varepsilon_2$, $\varepsilon_3 \in \C$ are 
roots of unity. Applying the Jordan decomposition of characters we
get:
\begin{eqnarray*}
\chi_{9,1}(t) \hspace{-0.2cm} & = & \hspace{-0.2cm} \frac{\varepsilon'_1}{6} ( 
  R_{w_9 w_2 w_3 w_2} (t) + 3 R_{w_9}(t) + 2 R_{w_9 w_2}(t) ),\\ 
\chi_{9,2}(t) \hspace{-0.2cm} & = & \hspace{-0.2cm}
\frac{\varepsilon'_2}{6} ( 2 R_{w_9 w_2 w_3 w_2}(t) - 2 R_{w_9 w_2}(t) ),\\ 
\chi_{9,3}(t) \hspace{-0.2cm} & = & \hspace{-0.2cm}
\frac{\varepsilon'_3}{6} ( R_{w_9 w_2 w_3 w_2}(t) - 3 R_{w_9}(t) + 2 R_{w_9 w_2}(t) ),
\end{eqnarray*}
where $\varepsilon'_1$, $\varepsilon'_2$, $\varepsilon'_3 \in \C$ are 
roots of unity. Since $R_{w}(t)(1) = R_{w}(1)(1)$ and since we can
easily evaluate the Deligne-Lusztig characters $R_{w}(1)$ at the
identity element~$1$ with \CHEVIE\ we get 
$\varepsilon_1' = \varepsilon_2' = 1$, $\varepsilon_3' = -1$
and also ${\chi}_{9,j}(t)(1)$ for $j=1,2,3$.

\section{Decomposition numbers}\label{sec:decnumbers}

Let $q$ be an arbitrary prime power and $\ell$ an odd prime not
dividing $q$, but dividing the group orders 
$|G|=|G^*| = q^9(q^6-1)(q^4-1)(q^2-1)$. This section is devoted to the
proof of the $\ell$-modular decomposition numbers of the unipotent
characters for both $G=\SO_7(q)$ and $G^*=\Sp_6(q)$. 
Let $(K,\Ocal,k)$ be an $\ell$-modular splitting system for $G$, $G^*$
and all their subgroups. 

If $\ell \nmid q\pm 1$ then all unipotent blocks of $G$ and $G^*$ are
cyclic, and we may refer to \cite{FongSrinivasanTrees} for odd $q$,
and to \cite{WhiteSp6even} for even $q$ to obtain their decomposition
numbers. Note that Remark~\ref{rmk:even_iso_Sp} gives an explicit
isomorphism $\SO_7(q) \to \Sp_6(q)$ for even $q$. Therefore we only
need to consider odd primes $\ell$ dividing $q+1$ or $q-1$.
The case of $\ell \mid q-1$ is readily solved with the help of 
theorems by Puig, Gruber and Hiss, see Remark~\ref{rem:q-1}.
Hence, we get new results only in the case of $\ell \mid q+1$. This
case is the topic of this entire section.

\begin{remark}\label{rem:q-1}
The decomposition numbers for $\ell \mid q-1$ may be computed as follows:
By~\cite{FongSrinivasanClass} for odd $q$ (see 
also~\cite[Sections~2 and 6]{HissKessar}) and~\cite{WhiteSp6even} for
even $q$ the distribution of the ordinary unipotent characters of $G$
and $G^*$ into blocks coincides with the distribution into
Harish-Chandra series. More precisely: There are two unipotent
blocks: the principal block containing the unipotent characters in the
principal series, i.e.\ those whose symbol has defect~$1$, and one
block containing the unipotent characters ${[1,-,3]}$ and ${[-,1,3]}$
whose symbols have defect $3$. 

If $\ell > 3$ then $\ell$ does not divide the order of the Weyl group
and the decomposition numbers of $G$ and $G^*$ follow from a result 
by Puig~\cite{Puig}. If $\ell=3$ then we may employ
\cite[Theorem~4.13]{GeckHiss} (see also \cite{GruberHiss}) to
immediately infer the decomposition matrices of both blocks with the
help of the decomposition matrices of the general linear groups, see
\cite{JamesGL}. In fact, the decomposition numbers of the unipotent
characters of $G$ and $G^*$ for $\ell \mid q-1$ coincide and they
do not depend on whether $q$ is odd or even; hence they can be read
off from the data given in~\cite[Theorem~2.1]{WhiteSp6even}.
\end{remark}

\smallskip

\begin{hypothesis}
From now on until the end of this paper, we assume that $\ell$ is an
odd prime dividing $q+1$. 
\end{hypothesis}

\smallskip

The following theorem is the main result of this paper. We write 
$(q+1)_\ell$ for the largest power of $\ell$ dividing $q+1$.

\begin{theorem} \label{thm:decnumbers}
For all prime powers $q$ the $\ell$-modular decomposition numbers of
the unipotent characters in the principal block of $G=\SO_7(q)$ and
$G^*=\Sp_6(q)$ are given in Table~\ref{tab:decmatSO7Sp6}, where the
decomposition numbers $\alpha, \beta, \gamma$ are as follows:
  \begin{itemize}
    \item[(a)] If $(q+1)_\ell = 3$ then $\alpha =\beta=\gamma =1$.
    \item[(b)] If $(q+1)_\ell = 5$ then $\alpha = \beta = \gamma = 2$.
    \item[(c)] If $(q+1)_\ell > 5$ 
     then $\alpha =  \gamma = 2$ and $\beta = 3$.
  \end{itemize}
\end{theorem}

\begin{table}[h]
 \[
 \begin{array}{|c|cccccccccc|}
  \hline
  & \varphi_1 & \varphi_2 & \varphi_3 & \varphi_4 & \varphi_5 &
  \varphi_6 & \varphi_7 & \varphi_8 & \varphi_9 & \varphi_{10} \\
  & ps & ps & A_1 & [B_2,\eta] & A_1 \times A_1' & A_1' & [B_2,\St] &
  c & c & c \\\hline
  {[3,-,1]} & 1 & \multicolumn{1}{:c}{.} & . & . & . & . & . & . & . & . \\\cdashline{1-1}\cdashline{2-5}
  {[2,1,1]} & 1 & \multicolumn{1}{:c}{1} & . & \multicolumn{1}{c:}{.} & . & . & . & . & . & . \\
  {[-,3,1]} & 1 & \multicolumn{1}{:c}{.} & 1 & \multicolumn{1}{c:}{.} & . & . & . & . & . & . \\
  {[1,-,3]} & . & \multicolumn{1}{:c}{.} & . & \multicolumn{1}{c:}{1}
  & . & . & . & . & . & . \\\cdashline{1-1}\cdashline{3-6} 
  {[1,2,1]} & 1 & 1 & 1 & \multicolumn{1}{c:}{.} & \multicolumn{1}{c:}{1} & . & . & . & . &
  . \\\cdashline{1-1}\cdashline{6-7} 
  {[1^2,1,1]} & 1 & 1 & . & . & \multicolumn{1}{c:}{1} & 1 &
  \multicolumn{1}{:c}{.} & . & . & . \\\cdashline{1-1}\cdashline{7-10} 
  {[1,1^2,1]} & 1 & 1 & 1 & \alpha & 1 & 1 & \multicolumn{1}{:c}{1} &
  . & \multicolumn{1}{c:}{.} & .\\ 
  {[1^3,-,1]} & 1 & . & . & . & . & 1 & \multicolumn{1}{:c}{.} & 1 &
  \multicolumn{1}{c:}{.} & .\\ 
  {[-,1,3]} & . & . & . & 1 & . & . & \multicolumn{1}{:c}{.} & . &
  \multicolumn{1}{c:}{1} & .\\\cdashline{1-1}\cdashline{8-11} 
  {[-,1^3,1]} & 1 & . & 1 & \alpha & 1 & 1 & 1 & \beta & \multicolumn{1}{c:}{\gamma} & 1\\
  \hline
 \end{array}
 \]
 \caption{The $\ell$-modular decomposition numbers of the unipotent
   characters in the principal blocks of $\SO_7(q)$ and $\Sp_6(q)$.}
 \label{tab:decmatSO7Sp6}
\end{table}

\begin{remark} \label{rmk:decnosSO7}
\begin{enumerate}
\item[(a)] For $(q+1)_\ell>5$ the decomposition numbers in
  Table~\ref{tab:decmatSO7Sp6} and the statement in part (c) of 
  Theorem~\ref{thm:decnumbers} were obtained independently by Olivier
  Dudas and Gunter Malle (private communication).

\item[(b)] For $G^*=\Sp_6(q)$ the decomposition numbers in
  Table~\ref{tab:decmatSO7Sp6} were already computed by An,
  Hiss~\cite{AnHissSteinberg} and K\"ohler~\cite{DissKoehler} for odd
  $q$ and White~\cite{WhiteSp6even} for even $q$ except for the
  entries $\beta$ and $\gamma$, for which they prove lower and upper 
  bounds.

\item[(c)] By~\cite{FongSrinivasanClass} and~\cite{WhiteSp6even} the
  characters ${[21,-,1]}$ and ${[-,21,1]}$ are contained in a cyclic
  block and do not belong to the principal block; see
  \cite{FongSrinivasanTrees}, \cite{WhiteSp6even} for the Brauer tree
  of this cyclic block.

\item[(d)] For each of the groups $G=\SO_7(q)$ and $G^*=\Sp_6(q)$
  there are ten irreducible Brauer characters $\varphi_1$,
  $\varphi_2$, \dots, $\varphi_{10}$ in the principal block. The
  second row of Table~\ref{tab:decmatSO7Sp6} lists for each   
  $\varphi_j$ the modular Harish-Chandra series containing~$\varphi_j$;  
  see~\cite[Section~2]{GeckHiss} for a definition of
  modular Harish-Chandra series. 

  We write $ps$ for the principal series and $c$ for cuspidal Brauer
  characters. The standard Levi subgroups corresponding to the sets
  $\{\alpha_1^{(*)}\}$, $\{\alpha_2^{(*)}\}$ and 
  $\{\alpha_1^{(*)}, \alpha_3^{(*)}\}$ of simple roots each have
  a unique cuspidal unipotent Brauer character, namely the modular
  Steinberg character. We denote the corresponding modular
  Harish-Chandra series of $G$ and $G^*$ by~$A_1$, $A_1'$ and $A_1 \times A_1'$. 

  Furthermore, we will see in Section~\ref{subsec:decSO5} that the
  Levi subgroup $L^{(*)}$ has exactly two cuspidal unipotent Brauer
  characters: the restriction of the cuspidal ordinary unipotent
  character to the $\ell$-regular elements and the modular
  Steinberg character. We write $[B_2,\eta]$ and $[B_2,\St]$ for the
  corresponding modular Harish-Chandra series of $G$ and $G^*$.

\item[(e)] For each of the groups $G=\SO_7(q)$ and $G^*=\Sp_6(q)$ the
  set of ordinary unipotent characters is partitioned into six
  families; see the information available in the \GAP-part of
  \CHEVIE~\cite{GAP4}, \cite{CHEVIE}. The distribution into families
  is indicated by the dotted lines in the leftmost column of 
  Table~\ref{tab:decmatSO7Sp6}.

\item[(f)] Theorem~\ref{thm:decnumbers}
  confirms~\cite[Conjecture~3.4]{GeckHiss} in the special
  case of $\SO_7(q)$ and $\Sp_6(q)$ for all prime powers $q$. We will
  see later in Theorem~\ref{thm:decmatSO5} that the conjecture also
  holds for the groups $\SO_5(q)$ and all prime powers $q$.
\end{enumerate}
\end{remark}

\medskip

In Section~\ref{subsec:rels} we derive some relations for the Brauer
characters of $G$ and~$G^*$. Section~\ref{subsec:cycblocks} gives some
cyclic blocks of the parabolic subgroups $P$ and~$P^*$. 
In order to find a good initial approximation for the decomposition
matrices of $G$ and $G^*$, we compute as a first step the decomposition matrix of
$\SO_5(q)$ in Section~\ref{subsec:decSO5}. It turns out that it coincides with the 
decomposition matrix of $\Sp_4(q)$ in \cite{OkuyamaWaki}. 
The approximation is given in Section~\ref{subsec:approx}, 
in which we also show that almost all entries are decomposition
numbers. Finally, in Section~\ref{subsec:proof} we dispel the last
ambiguities in the decomposition matrices and finish the proof of
Theorem~\ref{thm:decnumbers}. 

To eliminate these ambiguities we apply two approaches: for
$(q+1)_\ell > 5$ we make use of a recent result by Dudas (see
\cite{Dudas}) exploiting some deep results of Deligne-Lusztig theory
to obtain information on projective modules for $G$ and $G^*$ which
give upper bounds for the decomposition numbers. By providing suitable
relations of Brauer characters, we can show that these bounds are
tight. For small $(q+1)_\ell$ however, we need to find better upper
bounds than the ones Deligne-Lusztig theory provides. Hence for
$(q+1)_\ell \le 5$ we consider the above-mentioned cyclic blocks of
the parabolic subgroups $P$ and $P^*$. Again, we show that the
upper bounds thus obtained are tight by providing suitable character relations. 

\subsection{Relations}\label{subsec:rels}

In this section we always assume that $q$ is an odd prime power. 
Let $H$ be one of the groups $G = \SO_7(q)$, $G^* = \Sp_6(q)$,
$\widetilde{G}^* = \CSp_6(q)$, $G_5 = \SO_5(q)$ or $G_5^* = \Sp_4(q)$. 
For a class function $\psi$ of $H$ we write $\breve{\psi}$ for the
restriction of $\psi$ to the $\ell$-regular elements of $H$. Since
$\ell$ is odd the prime $\ell$ is good for the groups $\G$, $\G^*$ and
$\widetilde{\G}^*$ in the sense of~\cite[p.~28]{Carter}. A
general result of Geck and Hiss~\cite{GeckHissBasicSets} implies that  
$\{\breve{\chi} \mid \chi \in \Irr(H) \,\, \text{unipotent}\}$ is a
basic set for the unipotent blocks of~$H$. In this section we write
certain Brauer characters as $\Z$-linear combinations of the above basic
sets. The relations obtained in this way will be used in
Sections~\ref{subsec:decSO5}-\ref{subsec:proof} to determine lower
bounds for certain decomposition numbers of $H$. We denote a unipotent
character of $H$ by the bipartition labeling it; see Section~\ref{subsec:charsSpSO}.

\begin{lemma} \label{la:relsfromDLchars}
Let $q$ be odd. 
\begin{enumerate}
\item[(a)] There are ordinary characters $\chi_1, \chi_2$
  of~$\SO_5(q)$ such that 
  \begin{eqnarray*}
  \breve{\chi}_1 & = & -{[2,-,1]}\,\breve{} + {[-,-,3]} \,\breve{} + {[-,2,1]} \,\breve{},\\
  \breve{\chi}_2 & = & - {[-,-,3]} \,\breve{} -{[1^2,-,1]} \,\breve{} + {[-,1^2,1]} \,\breve{}.
  \end{eqnarray*}

\item[(b)] If $(q+1)_\ell > 3$ then there is an ordinary character
  $\chi_3$ of~$\SO_5(q)$ such that
  \begin{eqnarray*}
  \breve{\chi}_3 & = & {[2,-,1]}\,\breve{} -2 \cdot {[-,-,3]}\,\breve{} -{[1^2,-,1]}\,\breve{}   -
  {[-,2,1]}\,\breve{} + {[-,1^2,1]}\,\breve{}.
  \end{eqnarray*}

\item[(c)] Let $H \in \{\SO_7(q), \Sp_6(q)\}$. There are ordinary
  characters $\chi_{9,i}$ and $\chi_{13,j}$ of $H$ for $i=1,2,3$ and 
  $j=1,2,3,4$ such that 
  \begin{eqnarray*} 
  \breve{\chi}_{9,1} & = & {[3,-,1]}\,\breve{} - {[2,1,1]}\,\breve{} - 
  {[-,3,1]}\,\breve{} + {[1,2,1]}\,\breve{},\\
  \breve{\chi}_{9,2} & = & -{[-,3,1]} \,\breve{} - {[1,-,3]}\,\breve{} + 
  {[1,2,1]}\,\breve{} - {[1^2,1,1]}\,\breve{} + 
  {[1^3,-,1]}\,\breve{} + {[-,1,3]}\,\breve{},\\
  \breve{\chi}_{9,3} & = & {[1^2,1,1]}\,\breve{} - {[1,1^2,1]}\,\breve{} - 
  {[1^3,-,1]}\,\breve{} + {[-,1^3,1]}\,\breve{},\\
  \breve{\chi}_{13,1} & = & -{[3,-,1]}\,\breve{} +{[-,3,1]} \,\breve{}
  +{[1,-,3]}\,\breve{}-{[1,2,1]} \,\breve{} +{[1^2,1,1]}\,\breve{},\\
  \breve{\chi}_{13,2} & = & -{[-,3,1]} \,\breve{} - {[1,-,3]}\,\breve{} -
  {[1^2,1,1]}\,\breve{} + {[1,1^2,1]} \,\breve{} + {[1^3,-,1]}\,\breve{},\\ 
  \breve{\chi}_{13,3} & = & -{[2,1,1]} \,\breve{} - {[-,3,1]} \,\breve{} +
  {[1,2,1]} \,\breve{} + {[1^3,-,1]} \,\breve{} + {[-,1,3]}\,\breve{},\\
  \breve{\chi}_{13,4} & = & -{[1,2,1]}\,\breve{} +{[1^2,1,1]} \,\breve{}
  -{[1^3,-,1]} \,\breve{} -{[-,1,3]}\,\breve +{[-,1^3,1]}\,\breve{}.
  \end{eqnarray*}

\item[(d)] Let $H \in \{\SO_7(q), \Sp_6(q)\}$. If $(q+1)_\ell > 3$
  then there are ordinary characters $\chi_{23,1}, \chi_{23,2}$ of $H$
  such that
  \begin{eqnarray*} 
  \breve{\chi}_{23,1} & = &
  {[3,-,1]}\,\breve{} -2 \cdot {[-,3,1]}\,\breve{} -2 \cdot {[1,-,3]}\,\breve{}
  +{[1,2,1]}\,\breve{}-2 \cdot {[1^2,1,1]}\,\breve{} \\
  && \quad +{[1,1^2,1]}\,\breve{} +
  {[1^3,-,1]}\,\breve{},\\
  \breve{\chi}_{23,2} & = &
  {[2,1,1]}\,\breve{} + {[-,3,1]}\,\breve{} -2 \cdot {[1,2,1]}\,\breve{} + 
  {[1^2,1,1]}\,\breve{}-2 \cdot {[1^3,-,1]}\,\breve{}\\
  && \quad -2 \cdot {[-,1,3]}\,\breve{} 
  +{[-,1^3,1]}\,\breve{}.
  \end{eqnarray*}

\item[(e)] Let $H \in \{\SO_7(q), \Sp_6(q)\}$. If $(q+1)_\ell > 5$
  then there is an ordinary character $\chi_{32}$ of~$H$ such that
 \begin{eqnarray*}
  \breve{\chi}_{32} & = & 
  -{[3,-,1]}\,\breve{} + {[2,1,1]}\,\breve{} +3 \cdot {[-,3,1]}\,\breve{} + 2 \cdot {[1,-,3]}\,\breve{} -
  3 \cdot {[1,2,1]}\,\breve{} \\
  && +3 \cdot {[1^2,1,1]}\,\breve{} - {[1,1^2,1]}\,\breve{} - 3 \cdot
  {[1^3,-,1]}\,\breve{} - 2 \cdot {[-,1,3]}\,\breve{} + {[-,1^3,1]}\,\breve{}.
 \end{eqnarray*}
 \end{enumerate}
\end{lemma}
\begin{proof}
(a): Let $H := \SO_5(q)$. In the notation of Section~\ref{subsec:charsSpSO} 
  there is $t_2 \in \F^\times$ with $t_2^{q+1}=1$, $t_2 \neq \pm 1$ such
  that $g(1,t_2)$ is an $\ell$-element. Now we invoke 
  \cite[Proposition~12.6]{DigneMichel91}: Let $f$ denote the
  characteristic function on $\ell$-regular classes, i.e.\ for $x \in H$ 
  we set $f(x):=1$ if and only if $\ell \nmid |x|$, and $f(x):=0$ otherwise. 
  Note that $f \in \mathcal{C}(G)_{p'}$, i.e.\ is $p$-constant as
  $\ell \neq p$, so $f(x) = f(x_{p'})$. We therefore obtain that 
  $R_w(1,t_2)\,\breve{} = R_w(1)\,\breve{}$ for $w \in \{w_{212}, w_{212}w_1\}$.
  Using \CHEVIE\ we decompose $R_{w_{212}}(1)\,\breve{}$,
  $R_{w_{212}w_1}(1)\,\breve{}$ into a $\Z$-linear combination of    
  $\{\,\breve{\chi} \mid \chi \in \Irr(H) \,\, \text{unipotent}\}$. 
  It follows from~\eqref{eq:irr12SO5} and from 
  $\varepsilon_1' = \varepsilon_2' = -1$ that 
  $\chi_1 := \chi_1(1,t_2)$ and $\chi_2 := \chi_2(1,t_2)$ satisfy the
  equations in (a).

\smallskip

\noindent (b): If $(q+1)_\ell > 3$ then there are $t_1, t_2 \in \F^\times$ 
  with $t_1^{q+1} = t_2^{q+1} = 1$, $t_1,t_2 \neq \pm 1$ and 
  $t_2 \neq t_1^{\pm1}$ such that $g_\reg(t_1, t_2)$ is an
  $\ell$-element. In the same way as in (a) we can derive from
  \eqref{eq:irrregSO5} that $\chi_3 := \chi_3(t_1,t_2)$ satisfies the
  equation in (b).

\smallskip

\noindent (c),(d),(e): Suppose first that $H = \SO_7(q)$. As in the
proof of (a), (b) we can use equations~\eqref{eq:chi9_SO7}-\eqref{eq:chi32_SO7}, 
\cite[Proposition~12.6]{DigneMichel91} and \CHEVIE\ to see that
$\chi_{9,j} := \chi_{9,j}(t)$ for $j=1,2,3$, 
$\chi_{13,j} := \chi_{13,j}(t_1,t_2)$ for $j=1,2,3,4$, 
$\chi_{23,j} := \chi_{23,j}(t_1,t_2,t_3)$ for $j=1,2$ and
$\chi_{32} := \chi_{32}(t_1,t_2,t_3)$ for suitable choices
of the field elements $t$,~$t_1$,~$t_2$,~$t_3$ satisfy the 
equations in (c),(d),(e). 

Now suppose that $H = \Sp_6(q)$ so that $H = G^*$ is a subgroup of 
$\widetilde{G}^* = \CSp_6(q)$. The character table of
$\widetilde{G}^*$ was computed by L\"ubeck~\cite{DissLuebeck} and is
contained in the \CHEVIE\ library. It follows from~\cite[p.~140]{DigneMichel91},
\cite[Theorem~11.12]{BonnafeSLn} and from the degrees of the unipotent
characters of $G^*$ and $\widetilde{G}^*$ that restriction induces a
bijection between the set of unipotent characters of $\widetilde{G}^*$
and the set of unipotent characters of $G^*$. If $\chi_\Lambda$ is a
unipotent character of $G^*$ with the label $\Lambda$ we denote its
unipotent extension to $\widetilde{G}^*$ also by $\chi_\Lambda$ or
just $\Lambda$. Using the notation from~\cite{DissLuebeck} and the
explicit knowledge of the character table of $\widetilde{G}^*$ we get 
\begin{eqnarray*}
(\chi_{9,1}(k_1,k_2)\smash\downarrow^{\widetilde{G}^*}_{G^*})\,\breve{} & = &
({[3,-,1]}\smash\downarrow^{\widetilde{G}^*}_{G^*})\,\breve{} - 
({[2,1,1]}\smash\downarrow^{\widetilde{G}^*}_{G^*})\,\breve{} - 
({[-,3,1]}\smash\downarrow^{\widetilde{G}^*}_{G^*})\,\breve{} + 
({[1,2,1]}\smash\downarrow^{\widetilde{G}^*}_{G^*})\,\breve{} \\
& = & {[3,-,1]}\,\breve{} - {[2,1,1]}\,\breve{} - 
{[-,3,1]}\,\breve{} + {[1,2,1]}\,\breve{}
\end{eqnarray*}
for a suitable choice of the parameters $k_1,k_2$. Thus, with this
choice of $k_1$ and $k_2$, the character $\chi_{9,1} := \chi_{9,1}(k_1,k_2)$ satisfies the first equation
in~(c). Similarly, $\chi_{9,j} := \chi_{9,j}(k_1,k_2)\smash\downarrow^{\widetilde{G}^*}_{G^*}$ for
$j=2,3$, $\chi_{13,j} := \chi_{13,j}(k_1,k_2)\smash\downarrow^{\widetilde{G^*}}_{G^*}$ for $j=1,2,3,4$, 
$\chi_{23,j} := \chi_{23,j}(k_1,k_2,k_3)\smash\downarrow^{\widetilde{G}^*}_{G^*}$ for $j=1,2$, 
and finally $\chi_{32} := \chi_{32}(k_1,k_2,k_3,k_4)\smash\downarrow^{\widetilde{G}^*}_{G^*}$
for suitable parameters $k_1$, $k_2$, $k_3$, $k_4$ satisfy the remaining
equations in (c),(d),(e).
\end{proof}

\begin{remark}
The proof of Lemma~\ref{la:relsfromDLchars} is constructive in the
sense that it gives explicit descriptions of the ordinary characters
$\chi_i$ and $\chi_{i,j}$ on the left hand side of the equations in (a)-(e).
\end{remark}

\subsection{Cyclic blocks of parabolic subgroups}
\label{subsec:cycblocks}

In this section we determine certain cyclic blocks of the maximal 
parabolic subgroups $P = P_7$ of $G = \SO_7(q)$ and $P^* = P_6^*$ of 
$G^* = \Sp_6(q)$. Owing to the simple structure of their projective
indecomposable modules, these will serve our purposes twofold: via
induction they provide a source of projective modules for
an approximation of the decomposition matrices of their overgroups, while also
simultaneously giving small bounds for the decomposition numbers. 

In Section~\ref{subsec:decSO5} we calculate the decomposition
numbers of the principal block of $G_5=\SO_5(q)$ to later use this
information for $G = \SO_7(q)$. Therefore, we begin this section by
determining cyclic blocks of the parabolic subgroup $P_5$. 

To enable a concise statement of the following assertions, we need to
introduce some notation. Recall from Section~\ref{subsec:charsparabolic}
that the inertia subgroup $I^-$ of $\lambda^-$ in $P_n$ decomposes as
$L^-_n \ltimes U_n$ where $L^-_n \cong \GO_{n-3}^-(q)$. If
$n=5$ we have $L_5^- \cong \GO_{2}^-(q)$, which is a dihedral group of 
order $2(q+1)$; see \cite[Theorem~11.4]{Taylor}. For odd~$q$, let
$\nu_1^-$ 
be the non-trivial linear character of $L_5^-$ with 
$\SO_2^-(q) \le \ker(\nu_1^-)$; see~\cite[Remark~7.1]{HimNoeResSO} for
additional information. If~$q$ is even then~$L^{-}_5$ only possesses
two linear characters, and in this case we write~$\nu_1^-$ for the
unique non-trivial linear character of $L^{-}_5$. 

\begin{lemma} \label{la:cyclicblocksSO5}
Let $P_5$ be the parabolic subgroup of $G_5 = \SO_5(q)$ defined in
Section~\ref{subsec:parabolicsubgrps}. 
For all prime powers $q$ there is a cyclic block $b_5$ of $P_5$
whose Brauer tree is given in Table~\ref{tab:Brauertrees} with 
  \[ 
   \xi_1 := {^-}\psi_{1_{L_5^-}}, \text{~}
   \xi_2 := {^-}\psi_{\nu_1^-}, \text{ and }
   \xi_\mathrm{exc} := {^-}\psi_{\Xi},
  \]
  where $\Xi$ is the sum of all irreducible characters of degree $2$
  in the principal block of $L_5^-$. 
\end{lemma}
\setlength{\unitlength}{1cm}
\begin{table}[h]
 \centering
 \begin{tabular}{c}
  \begin{picture}(8,1.2)
    \multiput(3,0.5)(1.3,0){3}{\circle*{0.15}}
    \multiput(3,0.5)(1.3,0){2}{\line(1,0){1.3}}
    \put(4.3,0.5){\circle{0.3}}
    \put(2.8,0.8){$\xi_1$}
    \put(4.1,0.8){$\xi_\mathrm{exc}$}
    \put(5.4,0.8){$\xi_2$}
   \end{picture}
 \end{tabular}
 \caption{The Brauer tree of cyclic blocks of $P_5$, $P_7$, and $P_6^*$.}
 \label{tab:Brauertrees}
\end{table}
\begin{proof}
This is an application of Fong-Reynolds correspondence. 
Let $\xi \in \Irr(U_5)$, then $\xi$ is the only ordinary character in its
block $b$, since $U_5$ is an $\ell'$-group. By \cite[2B]{Fong1961} there
is a bijection of blocks, preserving defect groups and decomposition
matrices, between the blocks $\mathrm{Bl}(P_5 \mid b)$ covering $b$,
and the blocks of the inertia subgroup of $b$ covering $b$. We
take $\xi := \lambda^-$ as in Section~\ref{subsec:charsparabolic}
and write~$b^-$ for the block of~$U_5$ containing $\lambda^-$. Thus we
have a bijection between the sets $\mathrm{Bl}(P | b^-)$ and 
$\mathrm{Bl}(I^- | b^-)$. 

As $I^- \cong L_5^- \ltimes U_5$ is a split extension and $\lambda^-$
extends to $I^-$, \cite[2D]{Fong1961} gives a bijection between 
$\mathrm{Bl}(I^- | b^-)$ and the blocks of $L_5^- \cong D_{2(q+1)}$
preserving defect groups and decomposition matrices. As this
correspondence is realized via the character constructions detailed 
in Section~\ref{subsec:charsparabolic}, the claim follows from the
well-known blocks of the dihedral group $L_{5}^{-} \cong D_{2(q+1)}$.
\end{proof}

For the symplectic group $G^*$ and odd $q$, we follow
\cite{AnHissUnipotent} notationwise: we define the ordinary
irreducible characters $\nu_7$ and $\nu_8$ of the group $\Sp_2(q) = \SL_2(q)$ of
degree $(q-1)/2$ in the same way as~\cite[Lemma~5.1]{AnHissUnipotent}.

\begin{lemma} \label{la:cyclicblocks}
Let $P = P_7$ be the parabolic subgroup of $G = \SO_7(q)$ and 
$P^* = P_6^*$ the parabolic subgroup of $G^* = \Sp_6(q)$ defined in
Section~\ref{subsec:parabolicsubgrps}. 
\begin{itemize} 
\item[(a)] For all prime powers $q$ there is a cyclic block $b$ of $P$
  whose Brauer tree is given in Table~\ref{tab:Brauertrees} with
   \[
   \xi_1 := {^0}\psi_{({^-}\psi_{1_{L^{-}_5}})}, \text{~}
   \xi_2 := {^0}\psi_{({^-}\psi_{\nu_1^-})}, \text{ and }
   \xi_\mathrm{exc} := {^0}\psi_{({^-}\psi_{\Xi})},
   \]
   where $\Xi$ is the sum of all irreducible characters of degree $2$
   in the principal block of $L^{-}_5$.

\item[(b)] If $q$ is odd then there is a cyclic block $b^*$ of $P^*$
  whose Brauer tree is given in Table~\ref{tab:Brauertrees} with
  \[ 
   \xi_1 := {^2}\psi_{({^3}\psi^{1,+}_{\nu_7})}, \text{~}
   \xi_2 := {^2}\psi_{({^3}\psi^{1,+}_{\nu_8})}, \text{ and }
   \xi_\mathrm{exc} := {^2}\psi_{({^3}\psi^{1,+}_{\Xi})},
  \] 
  where $\Xi$ is the sum of all irreducible characters of degree $q-1$
  lying in the quasi-isolated block of $\Sp_2(q) = \SL_2(q)$
  containing $\nu_7$ and $\nu_8$.
 \end{itemize}
\end{lemma}
\begin{proof}
 We give the proof for the block $b$ in part~(a). 
 Let $\lambda^0$ be the non-trivial
 linear character of $U$ defined in Section~\ref{subsec:charsparabolic}. 
 Analogously to the proof of Lemma~\ref{la:cyclicblocksSO5},
 Fong-Reynolds correspondence gives a bijection between the blocks of
 $P$ covering the block $b^0:=\{\lambda^0\}$ of $U$ and the blocks of
 $P_{5}$, also preserving defect groups and decomposition matrices.
 We may repeat the same arguments for $P_{5}$ and 
 $\lambda^- \in \Irr(U_{5})$ to extend
 this correspondence to $L_5^-$. Again, this bijection is realized via
 the character constructions of Section~\ref{subsec:charsparabolic},
 so the claim follows.

 The existence of the block $b^*$ can be proved by the same arguments
 using the character constructions in~\cite[2.3]{AnHissSteinberg}.
\end{proof}

\subsection{The decomposition numbers of $\SO_5(q)$}
\label{subsec:decSO5}

In this section we apply our overall approach to
the group $G_5 = \SO_5(q)$ to determine the $\ell$-modular
decomposition numbers of the unipotent characters of~$G_5$. 

\begin{theorem} \label{thm:decmatSO5}
For all prime powers $q$ the $\ell$-modular decomposition numbers of
the unipotent characters in the principal block of $\SO_5(q)$ are given in
Table~\ref{tab:decmatSO5}, where the decomposition number $\alpha = 1$ 
if $(q+1)_\ell = 3$, and $\alpha = 2$ if $(q+1)_\ell > 3$.
\end{theorem}

\begin{table}[h]
 \[
 \begin{array}{|c|ccccc|}
  \hline
  & {}^5\varphi_1 & {}^5\varphi_2 & {}^5\varphi_3 & {}^5\varphi_4 & {}^5\varphi_5 \\
   & ps & c & \tilde{A}_1 & A_1 & c\\
  \hline
  {[2,-,1]} & 1 & \multicolumn{1}{:c}{.} & . & . & .
  \\\cdashline{1-1}\cdashline{2-5}
  {[-,-,3]} & . & \multicolumn{1}{:c}{1} & . & \multicolumn{1}{c:}{.} & .\\
  {[1^2,-,1]} & 1 & \multicolumn{1}{:c}{.} & 1 & \multicolumn{1}{c:}{.} & .\\
  {[-,2,1]} & 1 & \multicolumn{1}{:c}{.} & . & \multicolumn{1}{c:}{1} & . 
  \\\cdashline{1-1}\cdashline{3-6}
  {[-,1^2,1]} & 1 & \alpha & 1 & \multicolumn{1}{c:}{1} & 1\\
  \hline
 \end{array}
 \]
 \caption{The $\ell$-modular decomposition numbers of the unipotent
   characters in the principal block of $\SO_5(q)$.}
 \label{tab:decmatSO5}
\end{table}

\begin{remark} \label{rmk:decnosSO5}
\begin{enumerate}
\item[(a)] The character ${[1,1,1]}$ has defect $0$ and is not
  contained in the principal block.

\item[(b)] The set of ordinary unipotent characters of $\SO_5(q)$ is
  partitioned into three families; see the information available in
  the \GAP-part of \CHEVIE. The distribution into these families is
  indicated by the dotted lines in the leftmost column of
  Table~\ref{tab:decmatSO5}. 

\item[(c)] There are five irreducible Brauer characters ${}^5\varphi_1$,
  \dots, ${}^5\varphi_5$ in the principal block of $\SO_5(q)$ and each 
  belongs to a different Harish-Chandra series. We write $ps$ for the
  principal series and $c$ for cuspidal Brauer characters. A Levi
  subgroup corresponding to the short simple root $\alpha_1$ has a
  unique cuspidal unipotent Brauer character, the modular Steinberg
  character, and we denote the corresponding Harish-Chandra series of
  $\SO_5(q)$ by~$A_1$. The same is true for the Levi subgroup
  corresponding to the long simple root $\alpha_2$ and we write
  $\tilde{A}_1$ for the corresponding Harish-Chandra series of $\SO_5(q)$. 

\item[(d)] Theorem~\ref{thm:decmatSO5} confirms~\cite[Conjecture~3.4]{GeckHiss} 
  in the special case of $\SO_5(q)$ for all prime powers $q$.
\end{enumerate}
\end{remark}

\begin{proof}
Suppose that $q$ is odd. We begin by constructing several projective
characters of $G_5$ in the sense that these characters are sums of
the ordinary characters of projective indecomposable $kG_5$-modules (PIMs). 

Let $1_B$ be the trivial $kB$-module of a Borel
subgroup $B$ of the group $G_5$ and $\Phi_{1_B}$ the ordinary character of
a projective cover of $1_B$. We write $\varphi_\St$ for the modular
Steinberg module of a Levi subgroup $L_\mathrm{short}$ corresponding
to the short simple root $\alpha_1$ and let $\Phi_{\varphi_\St}$
be the ordinary character of a projective cover of $\varphi_\St$.
Similarly, we write $\tilde{\varphi}_{\St}$ for the modular Steinberg
module of a Levi subgroup $L_\mathrm{long}$ corresponding to the long
simple root $\alpha_2$ and let $\Phi_{\tilde{\varphi}_\St}$ be the
ordinary character of a projective cover of $\tilde{\varphi}_{\St}$. 
Furthermore, let $\Phi_{\nu_1^-}$ be the ordinary character of a
projective cover of the simple $kP_5$-module with Brauer
character $({^-}\psi_{\nu_1^-})\breve{}$ in the cyclic block $b_5$
of the parabolic subgroup~$P_5$ described in Lemma~\ref{la:cyclicblocksSO5}. 
Hence, in the notation of Lemma~\ref{la:cyclicblocksSO5} we have 
$\Phi_{\nu_1^-} = {^-}\psi_{\nu_1^-} + {^-}\psi_{\Xi}$. We define
characters of $G_5$ as follows:
\[
\Psi_1 := \R{G_5}{B}(\Phi_{1_B}), \quad 
\Psi_2 := \Phi_{\nu_1^-}\smash\uparrow_{P_5}^{G_5}, \quad
\Psi_3 := \R{G_5}{L_\mathrm{long}}(\Phi_{\tilde{\varphi}_{\St}}), \quad 
\Psi_4 := \R{G_5}{L_\mathrm{short}}(\Phi_{\varphi_\St}), 
\]
and let $\Psi_5$ denote the Gelfand-Graev character of $G_5$. The
characters $\Psi_1, \dots, \Psi_4$ are projective since induction 
and Harish-Chandra induction preserve projectives;
see~\cite[Lemma~4.4.3]{HissHabil}. The character $\Psi_5$ is
projective because it is induced from an $\ell'$-subgroup.

The scalar products of $\Psi_1$, \dots, $\Psi_5$ with the ordinary
unipotent characters of $G_5$ are given in Table~\ref{tab:scalSO5} and
can be determined as follows: Since Harish-Chandra induction commutes
with taking unipotent quotients (see~\cite[Lemma~6.1]{HissHC}) it is
straightforward to compute the scalar products of $\Psi_1$, $\Psi_3$,
$\Psi_4$ with the ordinary unipotent characters of $G_5$. The scalar
products of $\Psi_2$ follow from~\cite[Theorem~7.2]{HimNoeResSO} and
Frobenius reciprocity. By ~\cite[Section~12.1]{Carter} the
Steinberg character $\St_{G_5}$ is the only unipotent constituent of
$\Psi_5$ and it has multiplicity one. 

\begin{table}
 \[
 \begin{array}{|c|ccccc|}
  \hline
  & \Psi_1 & \Psi_2 & \Psi_3 & \Psi_4 & \Psi_5 \\
  \hline
  {[2,-,1]} & 1 & . & . & . & .\\
  {[-,-,3]} & . & 1 & . & . & .\\
  {[1^2,-,1]} & 1 & . & 1 & . & .\\
  {[-,2,1]} & 1 & . & . & 1 & . \\
  {[-,1^2,1]} & 1 & \frac{(q+1)_\ell-1}{2} & 1 & 1 & 1\\
  \hline
 \end{array}
 \]
 \caption{Scalar products of the unipotent characters of $G_5$ with
   the projective characters $\Psi_1$, $\Psi_2$, \dots, $\Psi_5$.}
 \label{tab:scalSO5}
\end{table}

From Table~\ref{tab:scalSO5} we get the unitriangular shape of the
decomposition matrix of the principal block $B_0(G_5)$ leading to a
bijection between the set of ordinary unipotent characters and the set
of irreducible unipotent Brauer characters of
$B_0(G_5)$. Let~${}^5\varphi_j$ be the irreducible Brauer character
corresponding to the $j$-th column of Table~\ref{tab:scalSO5} and let
${}^5\Phi_j$ be the ordinary character of its projective cover. In
particular, ${}^5\varphi_5$ is the modular Steinberg character of $G_5$.

The cuspidality of ${}^5\varphi_2$ and ${}^5\varphi_5$ is a consequence of 
\cite[Lemma~4.3]{HissHC} and \cite[Theorem~4.2]{GeckHissMalle}, respectively.
The Harish-Chandra series of ${}^5\varphi_1$, ${}^5\varphi_3$, ${}^5\varphi_4$
follow from the construction of $\Psi_1$, $\Psi_3$, $\Psi_4$. Thus,
all entries of Table~\ref{tab:scalSO5} are decomposition numbers
except for $\langle{[-,1^2,1]}, \Psi_2\rangle_{G_5}$.
Let $\alpha := \langle{[-,1^2,1]}, {}^5\Phi_2\rangle_{G_5}$. 
It follows from Lemma~\ref{la:relsfromDLchars} (a), (b) that 
$\alpha \ge 1$ and that $\alpha \ge 2$ if $(q+1)_\ell > 3$.
Hence, Table~\ref{tab:scalSO5} implies that $\alpha=1$ if
$(q+1)_\ell=3$. 

To prove that $\alpha \le 2$ in general, we use the method of Dudas
detailed in \cite{Dudas}. Let $S$ denote the simple $kG_5$-module
with Brauer character ${}^5\varphi_{5}$. 
Let $w$ be a Weyl group element of minimal length with the property that
the Deligne-Lusztig variety associated to $w$ gives rise to a
virtual module $P_w$ such that with respect to the usual pairing
$\left<{~},{~}\right>$ of simple and projective $kG_5$-modules we have
$\left< P_w,S\right> \neq 0$. Then by \cite[8.10,8.12]{BonnafeRouquier}
there is a bounded perfect complex \[ 0 \rightarrow Q_{\ell(w)}
\rightarrow Q_{\ell(w)+1} \rightarrow \cdots \rightarrow Q_{2\ell(w)}
\rightarrow 0 \] giving the homology with compact support on the
variety, and the projective cover of $S$ appears \textsl{only} in
$Q_{\ell(w)}$. The same holds for the unipotent summand of the complex whose
character is given by the Deligne-Lusztig character $\R{\G}{\T_w}(1)$,
so that we may verify the minimality property of $w$ by taking scalar
products of Deligne-Lusztig characters (which are readily available in the
\GAP-part of \CHEVIE) and 
\[
{}^5\varphi_5 = {[2,-,1]}\,\breve{} - \alpha{[-,-,3]}\,\breve{}
- {[1^2,-,1]}\,\breve{} - {[-,2,1]}\,\breve{}  + {[-,1^2,1]}\,\breve{}.
\]
For $w = w_1w_2$ we obtain the scalar product $2-\alpha$, while elements of
lesser length yield scalar product $0$. We assume that $\alpha \neq 2$
and consider the unipotent summand of the associated complex
$ 0 \rightarrow Q_2 \rightarrow Q_3 \rightarrow Q_4 \rightarrow 0 $
of projective $kG_5$-modules.
Owing to the wedge shape of the decomposition matrix, we see that 
$\R{\G}{\T_w}(1)$ is the unipotent part of 
${}^5\Phi_1 + {}^5\Phi_2 - {}^5\Phi_3 - {}^5\Phi_4 + (2-\alpha) \cdot {}^5\Phi_5$. 
By the above we have the information that $2-\alpha > 0$. Therefore,
$\alpha = 2$ or $\alpha < 2$, hence $\alpha \le 2$. 

Now suppose that $q$ is even. The ordinary character table of 
$\Sp_4(q)$ is available in the \CHEVIE\ library. The character
values on the unipotent classes imply that the isomorphism 
$\SO_5(q) \to \Sp_4(q)$ from Remark~\ref{rmk:even_iso_Sp} maps
each unipotent character with label $\Lambda$ to the unipotent
character with the same label, except for possibly swapping
$[1^2,-,1]$ and $[-,2,1]$. Hence, the decomposition numbers of the
unipotent characters of $G_5 = \SO_5(q)$ can be read off from those of
$\Sp_4(q)$ in~\cite[Theorem~2.2]{WhiteSp4even} and 
\cite[Theorem~2.3]{OkuyamaWaki}. The modular Harish-Chandra series can
be determined by the same arguments as for odd~$q$. 
\end{proof}

\subsection{Approximations}
\label{subsec:approx}

In this section we start with the determination of the decomposition
numbers of the groups $G = \SO_7(q)$ and $G^* = \Sp_6(q)$.
In \cite{AnHissSteinberg} an approximation of the decomposition matrix of
$G^*$ is constructed for odd $q$. Here we mainly follow the same
approach and show that similar arguments also provide an approximation
of the decomposition matrix of $G$ and $G^*$.

At the starting point of this construction lie the decomposition
matrices of the groups $G_5 = \SO_5(q)$ and $G^*_4 = \Sp_4(q)$; see
Theorem~\ref{thm:decmatSO5}, \cite[Theorem~4.2]{WhiteSp4odd}, 
\cite[Theorems 2.2, 5.1]{WhiteSp4even}, \cite[Theorem~2.3]{OkuyamaWaki}. 
The PIMs thus provided are a good source of projectives for $G$ and $G^*$
via Harish-Chandra induction. Further projective modules for $G$ and
$G^*$ can be obtained by inducing PIMs of the cyclic blocks of $P$ and
$P^*$ established in Lemma~\ref{la:cyclicblocks}.

\begin{lemma} \label{la:indproj}
We use the notation of Lemma~\ref{la:cyclicblocks} and set
$m_\mathrm{exp} := \frac{(q+1)_\ell-1}{2}$. Let $b$ and $b^*$ be the
cyclic blocks of $P$ and $P^*$, respectively, described in 
Lemma~\ref{la:cyclicblocks}. When dealing with the block $b^*$ we
assume that $q$ is odd. Let $\Phi_{\xi_j}$ be the ordinary 
character of a projective cover of the simple $kP^{(*)}$-module in
the block $b^{(*)}$ with Brauer character~$\breve{\xi}_j$ for $j=1,2$. 
Then the unipotent parts of $\Phi_{\xi_j}\smash\uparrow_{P^{(*)}}^{G^{(*)}}$ 
are 
\[
\begin{array}{ccccc}
  (\Phi_{\xi_1}\smash\uparrow_{P^{(*)}}^{G^{(*)}})_\mathrm{uni} & = &
  [1^3,-,1] & + & m_\mathrm{exp} \cdot [-,1^3,1], \\
  (\Phi_{\xi_2}\smash\uparrow_{P^{(*)}}^{G^{(*)}})_\mathrm{uni} & = & [-,1,3]
  & + & m_\mathrm{exp} \cdot [-,1^3,1].
\end{array}
\]
\end{lemma}
\begin{proof}
For the block $b$ this follows from~\cite[Theorem~8.1]{HimNoeResSO}; for 
$b^*$ it follows from \cite[Table~4]{AnHissUnipotent}, 
\cite[Corollary~3.3 and Example~3.5]{AnHissSteinberg}.
\end{proof}

The following lemma proves all modular Harish-Chandra series and
all decomposition numbers in Table~\ref{tab:decmatSO7Sp6} except for
$\beta$ and $\gamma$. The value of $\beta$ and $\gamma$ will be
determined in Section~\ref{subsec:proof}.

\begin{lemma} \label{la:approxdecmat}
For all prime powers $q$ the parts of the $\ell$-modular decomposition
matrices of $G$ and $G^*$ corresponding to the unipotent characters in
the principal block are given by Table~\ref{tab:decmatSO7Sp6}. The
same is true for the Harish-Chandra series of the irreducible Brauer
characters in the principal block of $G$ and $G^*$. The decomposition
numbers $\alpha$, $\beta$, $\gamma$ satisfy the following conditions:
\begin{enumerate}
 \item[(a)] $\alpha = 1$ if $(q+1)_\ell=3$, and $\alpha = 2$ if
   $(q+1)_\ell > 3$. 

 \item[(b)] $1\le \beta, \gamma \le \frac{1}{2}( (q+1)_\ell -1 )$.
  Furthermore, if
  $(q+1)_\ell > 3$ then $\beta, \gamma \ge 2$, and if $(q+1)_\ell > 5$ then
  $\beta \ge 3$.
\end{enumerate}
\end{lemma}

\begin{proof}
Suppose that $q$ is odd. The arguments for $G$ and $G^*$ are entirely
similar, allowing us to just state them for $G$. We start by
constructing several projective characters of $G$. 

The Levi subgroup $L$ of $P$ is isomorphic to 
$A \times G_5 \cong \F_q^\times \times \SO_5(q)$, hence we obtain PIMs of 
$L$ by inflating the PIMs of $G_5$. For $j=1,\dots,5$, we denote the
inflated PIM corresponding to the $j$-th column of
Table~\ref{tab:decmatSO5} by ${}^5\Phi_j$. Additionally, we set 
${}^5\Phi_6 := \infl_{G_5}^L([1,1,1])$. Let 
$L_{1,3} \cong G_3 \times \GL_2(q)$ be the standard Levi subgroup
corresponding to the set $\{\alpha_1, \alpha_3\}$ of simple roots and
let $\Phi_{\varphi_\St \boxtimes \varphi_\St'}$ be the ordinary character of a
projective cover of its modular Steinberg module. 

For $j=1,2$, let $\Phi_{\xi_j}$ be the ordinary character of a
projective cover of the simple $kP$-module with Brauer character
$\breve{\xi}_j$ in the cyclic block $b$ of~$P$ described in
Lemma~\ref{la:cyclicblocks}. Hence, in the notation of 
Lemma~\ref{la:cyclicblocks}~(a) we have: 
\[
\Phi_{\xi_1} = {^0}\psi_{({^-}\psi_{1_{L^{-}_5}})} + {^0}\psi_{({^-}\psi_{\Xi})} \quad
\text{and} \quad
\Phi_{\xi_2} = {^0}\psi_{({^-}\psi_{\nu_1^-})} + {^0}\psi_{({^-}\psi_{\Xi})}.
\]
We define characters of $G$ as follows:
\begin{align*}
\begin{split}
\Psi_1 &:= \R{G}{L}({}^5\Phi_1), \,
\Psi_2 := \R{G}{L}({}^5\Phi_6), \,
\Psi_3 := \R{G}{L}({}^5\Phi_4), \,
\Psi_4 := \R{G}{L}({}^5\Phi_2), \, \\
\Psi_5 &:= \R{G}{L_{1,3}}(\Phi_{\varphi_\St \boxtimes \varphi_\St'}), \,
\Psi_6 := \R{G}{L}({}^5\Phi_3), \,
\Psi_7 := \R{G}{L}({}^5\Phi_5), \, \Psi_8 := \Phi_{\xi_1}\smash\uparrow_P^G, \,
\Psi_9 := \Phi_{\xi_2}\smash\uparrow_P^G,
\end{split}
\end{align*}
and let $\Psi_{10}$ denote the Gelfand-Graev character of $G$. The
characters $\Psi_1, \dots, \Psi_{10}$ are projective since induction 
and Harish-Chandra induction preserve projectives;
see~\cite[Lemma~4.4.3]{HissHabil}. The character $\Psi_{10}$ is
projective because it is induced from an $\ell'$-subgroup.

The scalar products of $\Psi_1, \Psi_2, \dots, \Psi_{10}$ with the ordinary
unipotent characters of $G$ are given in Table~\ref{tab:scalSO7} and
can be determined as follows: Since Harish-Chandra induction commutes
with taking unipotent quotients (see~\cite[Lemma~6.1]{HissHC}) it is
straightforward to compute the scalar products of $\Psi_1, \Psi_2,
\dots, \Psi_7$ with the ordinary unipotent characters of $G$ with the
help of \CHEVIE. The scalar products of $\Psi_8$ and $\Psi_9$ follow
from Lemma~\ref{la:indproj}. By~\cite[Section~12.1]{Carter} the
Steinberg character $\St_G$ is the only unipotent constituent of
$\Psi_{10}$ and it has multiplicity one.  

\begin{table}
 \[
 \begin{array}{|c|cccccccccc|}
  \hline
  & \Psi_1 & \Psi_2 & \Psi_3 & \Psi_4 & \Psi_5 & \Psi_6 & \Psi_7 & \Psi_8 & \Psi_9 & \Psi_{10} \\
  \hline
  {[3,-,1]}   & 1 & . & . & . & . & . & . & . & . & . \\
  {[2,1,1]}   & 1 & 1 & . & . & . & . & . & . & . & . \\
  {[-,3,1]}   & 1 & . & 1 & . & . & . & . & . & . & . \\
  {[1,-,3]}   & . & . & . & 1 & . & . & . & . & . & . \\
  {[1,2,1]}   & 1 & 1 & 1 & . & 1 & . & . & . & . & . \\
  {[1^2,1,1]} & 1 & 1 & . & . & 1 & 1 & . & . & . & . \\
  {[1,1^2,1]} & 1 & 1 & 1 & \alpha & 1 & 1 & 1 & . & . & . \\
  {[1^3,-,1]} & 1 & . & . & . & . & 1 & . & 1 & . & . \\
  {[-,1,3]}   & . & . & . & 1 & . & . & . & . & 1 & . \\
  {[-,1^3,1]} & 1 & . & 1 & \alpha & 1 & 1 & 1 & \frac{(q+1)_\ell-1}{2} & \frac{(q+1)_\ell-1}{2} & 1 \\
  \hline
 \end{array}
 \]
 \caption{Scalar products of the unipotent characters of $G$ with
   the projective characters $\Psi_1$, $\Psi_2$, \dots, $\Psi_{10}$.}
 \label{tab:scalSO7}
\end{table}

From Table~\ref{tab:scalSO7} we get the unitriangular shape of the
decomposition matrix of the principal block $B_0(G)$ leading to a
bijection between the set of ordinary unipotent characters and the set
of irreducible unipotent Brauer characters of
$B_0(G)$. Let~$\varphi_j$ be the irreducible Brauer character
corresponding to the $j$-th column of Table~\ref{tab:scalSO7} and let
$\Phi_j$ be the ordinary character of its projective cover. In
particular, $\varphi_{10}$ is the modular Steinberg character of $G$.

The cuspidality of $\varphi_{10}$ is a consequence of 
\cite[Theorem~4.2]{GeckHissMalle}. We get the lower bounds for $\beta$
and $\gamma$ from the relations involving $\breve{\chi}_{9,3}$, 
$\breve{\chi}_{13,4}$, $\breve{\chi}_{23,2}$, $\breve{\chi}_{32}$ in
Lemma~\ref{la:relsfromDLchars}. Since $\Psi_7$ is Harish-Chandra
induced from a proper Levi subgroup and $\varphi_{10}$ is cuspidal we
get that $\langle{[-,1^3,1]}, \Phi_7\rangle_G = 1$. This proves all
statements about $\Phi_7, \Phi_8, \Phi_9, \Phi_{10}$ in
Lemma~\ref{la:approxdecmat}. Furthermore, we see that $\varphi_7$ is
the only irreducible Brauer character in the Harish-Chandra
series $[B_2,\St]$.
The construction of $\Psi_6$ and the fact that $\varphi_7$ belongs to
the Harish-Chandra series $[B_2,\St]$ imply that 
$\langle{[1,1^2,1]}, \Phi_6\rangle_G = 1$. From the relation involving
$\breve{\chi}_{9,2}$ in Lemma~\ref{la:relsfromDLchars} we get that 
$\langle{[1^3,-,1]}, \Phi_6\rangle_G = 1$. Again, the cuspidality of
$\varphi_{10}$ tells us that $\langle{[-,1^3,1]}, \Phi_6\rangle_G = 1$,
proving all statements about $\Phi_6$ and about the Harish-Chandra
series of $\varphi_6$. 

The construction of $\Psi_5$ and the Harish-Chandra series
of $\varphi_7$ and $\varphi_{10}$ imply all statements about
$\Phi_5$ in Lemma~\ref{la:approxdecmat}. Additionally, we see that 
$\varphi_5$ is the only irreducible Brauer character in the 
Harish-Chandra series $A_1 \times A_1'$.
From the modular Harish-Chandra series we get 
$\langle{[1,1^2,1]}, \Phi_4\rangle_G = \alpha$. Then the relation
involving $\breve{\chi}_{9,3}$ in Lemma~\ref{la:relsfromDLchars}
implies that $\langle{[-,1^3,1]}, \Phi_4\rangle_G = \alpha$. The lower
bound for $\gamma$ gives us $\langle{[-,1,3]}, \Phi_4\rangle_G = 1$, 
proving all statements about~$\Phi_4$. By the construction of
$\Psi_4$, the character $\varphi_4$ is the only irreducible Brauer
character in the Harish-Chandra series $[B_2,\eta]$. Again
from our knowledge of the modular Harish-Chandra series we get all
statements about $\Phi_3$, $\Phi_2$ and the Harish-Chandra series of
$\varphi_3$. 

The permutation character $1_P\smash\uparrow_P^G$ restricts to the
principal block as $[3,-,1]+[2,1,1]$. As~$\ell$ divides $[G:P]$, 
there is a trivial constituent in the $\ell$-modular reduction
$[2,1,1]\breve{~}$ so that $\langle{[2,1,1]}, \Phi_1\rangle_G = 1$.
Obviously, the trivial Brauer character $\varphi_1$ belongs to the
principal series. Now the modular Harish-Chandra series and the
relation involving $\breve{\chi}_{13,3}$ in Lemma~\ref{la:relsfromDLchars}
imply all statements about $\Phi_1$. Considering the character 
$\R{G}{T}(1_T)$ where $T := \T^F$ we see that the Harish-Chandra
series of $\varphi_2$, $\varphi_8$ and $\varphi_9$ are as
claimed. This completes the proof of Lemma~\ref{la:approxdecmat} for odd $q$.

Suppose that $q$ is even. The ordinary character table of 
$\Sp_6(q)$ is available in the \CHEVIE\ library. The character
values on the unipotent classes imply that the isomorphism 
$\SO_7(q) \to \Sp_6(q)$ from Remark~\ref{rmk:even_iso_Sp} maps
each unipotent character with label $\Lambda$ to the unipotent
character with the same label. Hence, the decomposition numbers of the
unipotent characters of $G = \SO_7(q)$ can be read off from those of
$\Sp_6(q)$ in~\cite[Theorem~2.2]{WhiteSp6even} except for the upper
bounds on $\beta$ and $\gamma$ in Lemma~\ref{la:approxdecmat} and we
only get $\langle{[2,1,1]}, \Phi_1\rangle_G \le 1$. In the same
way as for odd $q$ we get $\langle{[2,1,1]}, \Phi_1\rangle_G = 1$. 
By inducing the two PIMs in the cyclic block $b$ of $P$ from
Lemma~\ref{la:cyclicblocks} (b) we obtain the upper bounds for $\beta$
and $\gamma$. The modular Harish-Chandra series then follow by the
same arguments as for odd~$q$. 
\end{proof}

\subsection{Proof of Theorem~\ref{thm:decnumbers}}\label{subsec:proof}

In this final section we combine the information we have derived
previously and complete the proof of the decomposition numbers of $G=\SO_7(q)$
and $G^*=\Sp_6(q)$ given in
Theorem~\ref{thm:decnumbers}. According to Lemma~\ref{la:approxdecmat}
we only have to settle the ambiguities $\beta$ and $\gamma$. For
$(q+1)_\ell>5$ we employ again Dudas' method of \cite{Dudas} to show
that the lower bounds in Theorem~\ref{thm:decnumbers} are tight. If 
$(q+1)_\ell \in \{3, 5\}$, then Lemma~\ref{la:approxdecmat} already 
yields sufficient bounds. In this whole section, $q$ is an odd or
even prime power. 

\begin{lemma}
 \label{la:decnosgeneric}
 We have $\gamma\le2$ and $\beta\le\gamma+1$.
\end{lemma}
\begin{proof}
 Our approach follows the proof of Theorem~\ref{thm:decmatSO5}. We give the
 argument for the group $G$; the proof for $G^*$ may be copied verbatim.
 For a given element $w$ of the Weyl group $\W$ of $\G$, we consider the
 Deligne-Lusztig character $\R{\G}{\T_w}(1)$ and compute the scalar
 product with the approximated Brauer character 
 \[ \begin{split}
 \varphi_{10} = -{[3,-,1]}\,\breve{} + {[2,1,1]}\,\breve{} +
 \beta \cdot {[-,3,1]}\,\breve{} + \gamma \cdot {[1,-,3]}\,\breve{}
 -\beta \cdot {[1,2,1]}\,\breve{}\\ + \beta \cdot {[1^2,1,1]}\,\breve{}
 -{[1,1^2,1]}\,\breve{} - \beta \cdot{[1^3,-,1]}\,\breve{} -\gamma
 \cdot {[-,1,3]}\,\breve{} &+ {[-,1^3,1]}\,\breve{}.
\end{split}
 \]
We obtain the scalar products $2\gamma-4$ for $w=s_1s_2s_3=:w'$, and
$2\beta-2\gamma-2$ for $w=s_1s_2s_1s_2s_3=:w''$; elements of lesser
length yield scalar product $0$. Note that with the \GAP-part of
\CHEVIE\ we get:
\begin{eqnarray*}
  \R{\G}{\T_{w'}}(1) &=& {[3,-,1]}-{[2,1,1]}+{[1,-,3]}+{[1,1^2,1]}
  -{[-,1,3]} - {[-,1^3,1]},\\
  \R{\G}{\T_{w''}}(1) &=& {[3,-,1]} -{[21,-,1]}
  -{[1,-,3]}-{[1,2,1]}+{[1^2,1,1]}+{[-,21,1]}\\
 && +{[-,1,3]}-{[-,1^3,1]}.
\end{eqnarray*}
Owing to the wedge shape of the decomposition matrix, we see that the
restriction of $\R{\G}{\T_{w'}}(1)$ to the principal block $B_0(G)$
coincides with the unipotent part of
\begin{equation}
\Phi_1 -2\Phi_2 - \Phi_3 + \Phi_4 + 2 \Phi_5 - \Phi_6 + (2-\alpha)\Phi_7 
-2\Phi_9+2(\gamma-2)\Phi_{10},\label{eq:gamma}
\end{equation}
and that the restriction of $\R{\G}{\T_{w''}}(1)$ to $B_0(G)$ 
is the unipotent part of
\begin{equation}
\Phi_1-\Phi_2-\Phi_3-\Phi_4+\Phi_6+\alpha\Phi_7-2\Phi_8+
  2\Phi_9+2(\beta-\gamma-1)\Phi_{10}. \label{eq:beta}
\end{equation}
Therefore, either $\gamma=2$, or \eqref{eq:gamma} gives
$(2\gamma-4) < 0$, hence $\gamma \le 2$ as claimed.
Analogously, we obtain from \eqref{eq:beta} that either $\beta = \gamma+1$,
or $\beta < \gamma + 1$, giving the bound for $\beta$, too.
\end{proof}

\begin{corollary}
 \label{cor:decnosgeneric}
 If $(q+1)_{\ell} > 5$ then $\beta = 3$ and $\gamma = 2$.
\end{corollary}
\begin{proof}
 This follows from the lower bounds for $\beta$ and $\gamma$ given
 in Lemma~\ref{la:approxdecmat} and the upper bounds of 
 Lemma~\ref{la:decnosgeneric} above.
\end{proof}

We now treat the remaining two cases $(q+1)_\ell =3$ and $(q+1)_\ell = 5$, 
thereby completing the proof of Theorem~\ref{thm:decnumbers}.

\begin{corollary}
 \label{cor:decnossmall}
 If $(q+1)_\ell = 3$ then $\beta = \gamma = 1$, and if $(q+1)_\ell = 5$ then
 $\beta = \gamma = 2$.
\end{corollary}
\begin{proof}
 This is an immediate consequence of the bounds in
 Lemma~\ref{la:approxdecmat} (b). 
\end{proof}

\section*{Acknowledgement}
The authors thank Gerhard Hiss for many insightful discussions.


\newcommand{\etalchar}[1]{$^{#1}$}

\end{document}